\documentclass[10pt, reqno]{amsart}
\usepackage{amscd} 
\usepackage{amsfonts} 
\usepackage{amssymb} 
\usepackage{latexsym} 
\usepackage[arrow, matrix, curve]{xy}

\newcommand{\ncm}{\newcommand}

\newtheorem{theorem}{Theorem}[section]
\newtheorem{prop}[theorem]{Proposition}
\newtheorem{problem}[theorem]{Problem}
\newtheorem{lemma}[theorem]{Lemma}
\newtheorem{cor}[theorem]{Corollary}
\newtheorem{lem&def}[theorem]{Lemma \& Definition}
\newtheorem{definition}[theorem]{Definition}
\newtheorem{example}[theorem]{Example}
\newtheorem{remark}[theorem]{Remark}

\def\C{\mathbb{C}\,}

\def\N{\mathbb{N}\,}

\ncm{\Ann}{\mbox{\rm Ann}\,}
\ncm{\End}{\mbox{\rm End}\,}

\def\Hom{\mbox{\rm Hom}\,}
\def\Indec{\mbox{\rm Indec}\,}

\def\|{\, | \,}

\def\id{\mbox{\rm id}}

\def\into{\hookrightarrow}

\def\Ind{\mbox{\rm Ind}}

\def\bra{\langle}
\def\ket{\rangle}

\ncm{\rarr}[1]{\stackrel{#1}{\longrightarrow}}
\ncm{\larr}[1]{\stackrel{#1}{\longleftarrow}}
\def\cop{\Delta}

\def\eps{\varepsilon}

\def\-2{_{(-2)}}
\def\-1{_{(-1)}}
\def\0{_{(0)}}
\def\1{_{(1)}}
\def\2{_{(2)}}
\def\3{_{(3)}}

\def\du1{\hat 1}

\begin{document}
\title[Quotient Modules in Depth]{An in-Depth Look at Quotient Modules}
\author{Lars Kadison} 
\address{Department of Mathematics \\ University of Pennsylvania \\ David Rittenhouse Laboratory \\ 209 S.\ 33rd St. \\
Philadelphia, PA 19104 } 
\email{lkadison@math.upenn.edu} 
\date{} 

\begin{abstract}
The coset $G$-space of a finite group  and a subgroup  is a fundamental module of study of Schur and others around 1930; for example, its endomorphism algebra is a Hecke algebra of double cosets.  We study and review its generalization $Q$ to Hopf subalgebras, especially the tensor powers and similarity as modules over a Hopf algebra, or what's the same, Morita equivalence of the endomorphism algebras.  
We prove that $Q$ has a nonzero integral if and only if the modular function  restricts to the modular function of the Hopf subalgebra.  We also study and organize knowledge of $Q$ and its tensor powers in terms of annihilator ideals, sigma categories, trace ideals, Burnside ring formulas, and when dealing with semisimple Hopf algebras, the depth of $Q$ in terms of the McKay quiver and the Green ring.
\end{abstract} 
\maketitle

\section{Introduction and Preliminaries}
Normality of subgroups and Hopf subalgebras have been studied from  several points of views such as stability under the adjoint representation and equality of certain subsets under multiplication and its opposite.  About ten to twenty years ago, normality was extended to subrings succesfully by a ``depth two'' definition in \cite{KS} using only a tensor product of natural bimodules of the subring pair and module similarity in \cite{AF}.   See for example the paper
\cite{BK} for a key theorem and more background.  After a Galois theory of bialgebroids were associated with depth two subrings in \cite{KS}, some inescapable questions were making themselves known,  "What kind of normality is depth one?"  "What is a subring of depth $n \in \N$, and how weak a notion of normality is this?"  For example, for subgroups and their corresponding group algebra pairs, the answers to these questions are in \cite{BK2} and in \cite{BKK, BDK}, respectively, where it is also noted that subgroup depth of any finite-dimensional pair of group algebras is finite.    

The challenge in  extending the theory of subring depth from group theory to Hopf algebra theory is taken up in among others \cite{K2014, HKY, HKS, HKL}, mostly through a study of a generalization of the cocommutative coalgebra of the finite G-set of cosets.  The focus of these papers is reducing depth computation to that of considerations of tensor power properties of the quotient module coalgebra of a Hopf subalgebra pair, also known to Hopf algebra theorists inspired by  results in algebraic groups. When this module coalgebra $Q$ is viewed in the Green ring, the depth is finite if the corresponding element is algebraic; the depth is closely related to the degree of the minimum polynomial.  In general, it is unknown whether $Q$ is algebraic; although for $Q$ a permutation module over a group, Alperin's theorem that permutation modules are algebraic (see \cite[Feit, IX.3.2]{Fe}) provides a different proof of the result in \cite{BDK} that subgroup depth is finite.   

 This paper continues the study of  the quotient module $Q$; among other things, extending the fundamental theorem of Hopf modules along the lines of Ulbrich in Theorem~\ref{th-ulbrich} below, answering a question, implied in the paper  \cite{BFG} and  restricted to finite dimensions, in terms of a nonzero integral in $Q$ and (ordinary nontwisted) Frobenius extensions (see Theorem~\ref{theorem-answer}), studying a Mackey theory of $Q$ with labels allowing variation of Hopf-group subalgebra (Section~\ref{subsection-Mackey}), and the ascending chain of trace ideals of $Q$ and its tensor powers in Section~\ref{section-traceideal}. In Section~\ref{section-endoQ}, we point out that the endomorphism algebra of $Q$ is a generalized Hecke algebra, define and study from several points of view a tower of endomorphism algebras of increasing tensor powers of $Q$, which is an example of a topic of 
current study of endomorphism algebras of tensor powers of certain modules over various groups and quantum groups.  We show in Propositon~\ref{prop-conv} that the endomorphism algebras of the tensor powers of $Q$ need only be Morita equivalent for two different powers in general to answer the problem in
\cite[p.\ 259]{BDK}.
 In the final section, we examine $Q$ for semisimple Hopf algebra pairs over an algebraically closed field of characteristic zero, and note  relationships with topics of fusion theory such as the McKay quiver and Perron-Frobenius dimension.  
 \subsection{Preliminaries} 
For any ring $A$, and $A$-module $X$, let $1 \cdot X = X$, $2 \cdot X = X \oplus X$, etc.  The similarity relation $\sim$ is defined on $A$-modules as follows.  Two $A$-modules $X,Y$ are similar, written $X \sim Y$, if $X \| n \cdot Y$ and $Y \| m \cdot X$ for some positive integers $m, n$.  This is an equivalence relation, and carries over to isoclasses in the Grothendieck group of $A$, or the Green ring if $A$ is a Hopf algebra.  If $M \sim N$ and $X$ is an $A$-module,
then we have $M \oplus X \sim N \oplus X$; if $\otimes$ is a tensor on  the category of finite-dimensional modules denoted by mod-$A$, then also $M \otimes X \sim N \otimes X$. In case $A$ is a finite-dimensional algebra, $M \sim N$ if and only if $\Indec (M) = \Indec (N)$, where $\Indec (X)$ denotes the set of isoclasses of the indecomposable module constituents of $X$ in its Krull-Schmidt  decomposition. Also $\End M_A$ and $\End N_A$ are Morita equivalent algebras if  $M_A \sim N_A$ (with Morita context bimodules $\Hom (M_A, N_A)$, $\Hom (N_A, M_A)$ and composition of maps); if $M_A \| n \cdot N_A$ for some $n \in \N$, and $\End M_A$, $\End N_A$ are Morita equivalent algebras, then 
$M_A \sim N_A$, as noted below in Proposition~\ref{prop-conv}.  

To a ring extension $B \rightarrow A$, we consider the natural bimodules ${}_BA_B$, ${}_AA_B$
and ${}_BA_A$ (where the first is a restriction of the second or third) in tensor products
$A^{\otimes_B n} = A \otimes_B \cdots \otimes_B A$ ($n$ times $A$, integer $n \geq 1$).
The ring extension $A \| B$ is said to have \textit{h-depth} $2n-1$ (where $n \in \N$) if 
\begin{equation}
\label{eq: simdepth}
A^{\otimes_B n} \sim A^{\otimes_B (n+1)}
\end{equation}
 as $A$-$A$-bimodules (equivalently, $A^e$-modules) \cite{K2012}.  Let $d_h(B,A)$ denote the least such natural number $n$; $d_h(B,A) = \infty$ if there is no such $n$ where similarity of the tensor powers of
$A$ over $B$ holds.  Since for any ring extension $A^{\otimes_B n} \| A^{\otimes_B (n+1)}$ via the identity element and multiplication,
it suffices for h-depth $2n-1$ to check just one condition $A^{\otimes_B (n+1)} \| q \cdot A^{\otimes_B n}$ for some
$q \in \N$. 

\begin{example}
\begin{rm}
If $d_h(B,A) = 1$, then ${}_AA \otimes_B A_A \oplus * \cong q \cdot {}_AA_A$ for some $q \in \N$, the H-separability condition on a ring extension of Hirata (thus, the H and its lower case to avoid confusion with Hopf).  In fact, Hirata proves
this condition alone implies ${}_AA_A \oplus * \cong  {}_AA \otimes_B A$, the separability condition on a ring extension.  
\end{rm}
\end{example}

If this is applied to a Hopf subalgebra pair $R \subseteq H$, the tensor powers of $H$ over $R$
may be rewritten in terms of the tensor product in the finite tensor category mod-$H$ as follows,
\begin{equation}
\label{eq: tensor powers}
H^{\otimes_R (n+1)} \cong H \otimes Q^{\otimes n}
\end{equation}
where $Q^{\otimes n} := Q \otimes \cdots \otimes Q$ ($n$ times $Q$, $\otimes$ the tensor product in mod-$H$) 
(see Eq.~(\ref{eq: fundamental}) below and \cite[Prop.\ 3.6]{K2014}). It follows that the h-depth
$2n+1$ condition of a Hopf subalgebra pair is equivalently 
\begin{equation}
\label{eq: moduledepth}
Q^{\otimes n} \sim Q^{\otimes {n+1}}
\end{equation}
 the depth $n$ condition on a right $H$-module
coalgebra $Q$, where $d(Q_H)$ denotes the least such integer $n \geq 0$ (say $Q^0 = k_H$).
It follows that
\begin{equation}
\label{eq: jpaa}
d_h(R,H) = 2d(Q_H) + 1
\end{equation}
(see \cite[Theorem 5.1]{K2014} for details of the proof).  

Subgroup depth is defined in \cite{BDK} in any characteristic, and in \cite{BKK} over an algebraically closed field $k$ of characteristic zero. If  $G$ is a finite group with subgroup $H$, the minimum depth $d(H,G)$ is determined as the lesser value of two other minimum depths, the minimum even and odd depths. The minimum even depth, $d_{ev}(H,G)$ assumes even natural number values, and is determined  from the  bipartite graph of the inclusion  of the semisimple
group algebras $B = kH \subseteq A = kG$ using \cite[Theorem 3.10]{BKK}, or from $n \in \N$ satisfying Eq.~(\ref{eq: simdepth}) as $A$-$B$-bimodules (equivalently, $B$-$A$-bimodules, the depth $2n$ condition).  The minimum odd depth, $d_{odd}(H,G)$ assumes odd natural number values, and is determined from Eq.~(\ref{eq: simdepth}) viewed this time as $B$-$B$-bimodules (the depth $2n+1$ condition), or from the diameter of the white vertices labelled by the irreducible characters of the subgroup in the bipartite graph  of $H \leq G$ as explained in
\cite[Theorem 3.6]{BKK}.
Then $d(H,G) = \min \{ d_{ev}(H,G), d_{odd}(H,G) \} $.  
Subgroup depth is studied further with many examples in \cite{ D, F, FKR, HHP, HHP2} as well as theoretically in \cite{BDK, BKK, HKY, HKS}, and extended to Hopf subalgebra
pairs in \cite{BDK, BKK, HKY, HKS}. The minimum depth and h-depth
of a Hopf subalgebra pair $R \subseteq H$ (beware the change in $H$!) are closely related by 
\begin{equation}
| d(R,H) - d_h(R,H) | \leq 2
\end{equation}
and both are infinite if one is infinite 
(see \cite{K2012}).  In addition, the authors of \cite{BDK} show that $d(H,G)$ depends  only  on the characteristic of the ground field, and may be labelled accordingly.  Several results on $Q$, depth and normality generalize from Hopf subalgebras to
left coideal subalgebras  of a finite-dimensional Hopf algebra, as noted in \cite{HKS, HKL};  recent papers of Cohen and Westreich advocate this point of view in normality (depth two).  

\begin{example}
\begin{rm}
Note that $d(H,G) = 1$ if the corresponding group algebras satisfy ${}_BA_B \oplus * \cong m \cdot {}_BB_B$, equivalent to $A \cong B \otimes_{Z(B)} C_A(B)$, where $Z(B)$ denote the center of $B$ and $C_A(B)$, 
the centralizer of $B$ in $A$.  For this, $G = H C_G(H)$ is a sufficient condition,  in particular,  $H$ is normal in $G$ \cite{BK2}.
The conjugation action of $G$ on $Z(B)$ spanned by the sum of group elements in a conjugacy class, is computed immediately to be the identity action.  The converse may be proven as an exercise using  \cite[Theorem 1.8]{BK2}.
\begin{prop}
Suppose $B \subseteq A$ are a subalgebra pair of group algebras over a field of characteristic zero corresponding to a  subgroup pair $H \leq G$ where $|G| < \infty$.  The depth $d_0(H,G) = 1$ if and only if the adjoint action of $G$ on $Z(B)$ is the identity. 
\end{prop} 
\end{rm}
\end{example}

%%%%%%%%%%%%%%%%%%%%%%%%%%%%%%%%%%%%%%%%%%%%%

\section{Quotient modules, Integrals and Mackey Theory} 
Let $H$ be a finite-dimensional Hopf algebra over a field $k$.  Given a Hopf subalgebra $R$ of $H$,
let $R^+$ denote the elements of $R$ with zero counit value.  Define the (right) quotient module $Q^H_R := H/R^+H$ (or simply by $Q$ when the context provides an unchanging Hopf subalgebra pair)
which is a right $H$-module coalgebra (since $R^+H$ is a coideal in the coalgebra of $H$).  Denote
the elements of $Q^H_R$ by $Q^H_R = \{ \overline{h} := h + R^+H \| \, h \in H \}$.  Note the canonical epimorphism of right $H$-module coalgebras, $H \rightarrow Q^H_R \rightarrow 0$ given by
$h \mapsto \overline{h}$.   

For more about the quotient module, we refer to  \cite{K2014}, where it is noted
that $Q^H_R \cong k \otimes_R H$
is an $R$-relative projective $H$-module (see also \cite{HKL}), that from the Nichols-Zoeller Theorem,  $\dim Q^H_R = \dim H/ \dim R$, it 
is shown in  \cite[Theorem 3.5]{K2014} that $Q^H_R$ is a projective $H$-module iff  $R$ is semisimple (equivalently for Hopf algebras, separable $k$-algebra),  and that $Q^H_R \cong t_R H$ where $t_R$ is a right integral in $R$. Tensoring an $H$-module $M_H$ by the quotient module $Q^H_R$ is naturally isomorphic to restricting $M$ to $R$, then inducing to an $H$-module:
\begin{equation}
\label{eq: fundamental}
M \otimes_R H \stackrel{\cong}{\longrightarrow} M \otimes Q^H_R
\end{equation}
where the mapping is given by $m \otimes_R h \mapsto mh\1 \otimes \overline{h\2}$, and $\otimes$ denotes the tensor in the tensor category mod-$H$. 
See also \cite{HKY, HKS} for more on $Q^H_R$ and the relationship with depth and h-depth of $R$ in $H$, and extending results to the more general case when $R$ is a left coideal subalgebra of $H$.    

\begin{example}
\label{example-qha}
\begin{rm}
Let $R = k1_H$ and $M = H$ in Eq.~(\ref{eq: fundamental}): then 
\begin{equation}
\label{eq: depthone}
H \otimes_k H_H \cong H_. \otimes H_.
\end{equation}
 for any Hopf algebra $H$, where the righthand side has the diagonal action of $H$.  If $R$ is a normal
Hopf subalgebra of $H$, then $R^+ H = HR^+$ is an Hopf ideal, and so $Q = H/R^+H$ is a Hopf algebra.
It follows from Eq.~(\ref{eq: depthone}) that $Q^{\otimes 2}_H = Q_. \otimes Q_. \cong Q \otimes_k Q_H$ since $H \rightarrow Q$ is a Hopf algebra epimorphism.  But $Q \otimes_k Q_H \cong
(\dim Q) Q_H$, so that $Q^{\otimes 2 } \sim Q_H$, and $d(Q_H) \leq  1$. (Then $d_h(R,H) \leq 3$;
in fact, if $R \neq H$, $d_h(R,H) = 3$ \cite{K2014, HKL}. Note that $d(Q_R) = 0$
since $\overline{h}r = \overline{hr} = \overline{h}\eps(r)$ for each $h \in H, r \in R$.  Then
$d(R, H) \leq 2$.) 
\end{rm}
\end{example}  

\begin{example}
\label{example-group}
\begin{rm}
 Let $H = kG$ be a finite group algebra, $J \leq G$ a subgroup, and $R= kJ$  a Hopf subalgebra of $H$ obviously.  One computes that the quotient module coalgebra $Q^H_R$ in simplified notation $Q^G_J \cong k[J \setminus G]$,
the $k$-coalgebra on the set of right cosets of $J$ in $G$ (and right $H$-module) \cite[Example 3.4]{K2014}.  
\end{rm}
\end{example}

The main problem in the area of Hopf algebra depth is whether in general any one of $d(Q^H_R)$, $d_h(R,H)$,
or $d(R,H)$ is finite \cite{BDK} (which would imply the finiteness of the other two depths). In order to emphasize the point that $Q^H_R$ is an $R$-relative projective $H$-module, we may make the following definition, and prove the next proposition, which formally addresses this problem and extends \cite[Cor.\ 5.8(i)]{K2014}. Let $A \supseteq B$ be an algebra extension of finite-dimensional algebras. 
Recall that a module $V_A$ is \textit{($B$-)relative projective} if the multiplication epimorphism $V \otimes_B A \rightarrow V$ splits as $A$-modules.  For example, if $A= H$ and $B = R$, an $H$-module $V$ is relative projective if the epimorphism $V \otimes \eps_Q$ from $V \otimes Q \rightarrow V$ is $H$-split (by Eq.~(\ref{eq: fundamental})), which is applied to modular representation theory by Carlson and  the recent \cite{La}. 
\begin{definition}
An algebra extension $A \supseteq B$ is said to have \textbf{finite representation type (f.r.t.)} if  only finitely many isoclasses of indecomposable $A$-modules are $B$-relative projective. 
\end{definition}

For example, if either $B$ or $A$ has f.r.t., then the extension $A \supseteq B$ has
f.r.t. For $A$ having f.r.t., it follows immediately; for $B$ having f.r.t., it follows as an exercise from the Krull-Schmidt Theorem and the  fact that
a relative projective $W_A$ satisfies $W \| V \otimes_B A$ for some $B$-module $V$. A semisimple extension $A \supseteq B$ has f.r.t if and only if  $A$  has f.r.t., since every $A$-module is relative projective (cf.\ \cite[Section 3.2]{K2016}).   

For $A= H$ and $B = R$ a finite-dimensional Hopf subalgebra pair, the relative projectives form an ideal $A(H,R)$ in the complex Green ring $A(H)$ of $H$ (see the proof below, Eq.~(\ref{eq: fundamental}) and \cite[II.7]{L}).  Then the Hopf algebra extension $H \supseteq R$ has f.r.t. if and only if $A(H,R)$ is finite-dimensional. 
\begin{prop}
If $H \supseteq R$ is a Hopf algebra extension (as above) having finite representation type, then the depth $d(R,H) < \infty$. 
\end{prop}
\begin{proof}
In fact as $H$-modules, $Q = H/R^+H \cong k \otimes_R H$. Then the tensor powers of $Q$  are relative projectives, by an induction argument on the power using $Q^{\otimes (n+1)} \cong Q^{\otimes n} \otimes_R H$.  Since $Q^{\otimes m} \| Q^{\otimes (m+1)}$ for each $m \in \N$, the indecomposable constituents of $Q$ and its tensor powers satisfy $\Indec(Q^{\otimes m}) \subseteq \Indec (Q^{\otimes (m+1)})$
within  a finite set, so that $Q^{\otimes n} \sim Q^{\otimes (n+1)}$ for some $n \in \N$. Then
$d(Q_H) < \infty$.  
\end{proof}

The following proposition is known, but the proof and upper bound are somewhat new.  
\begin{prop}
If $k_R \| R_R$, then $d(Q_H) \leq N +1$ where $N$ is the number of nonisomorphic principle $H$-modules. Consequently, the h-depth $d_h(R,H) \leq 2N+3$. 
\end{prop}
\begin{proof}
(Recall the classic result that $k$ is a direct summand of a Hopf algebra $R$ as right $R$-modules if and only if
$R$ is a semisimple algebra.) Tensoring $k_R \| R_R$ by $-\otimes_R H$ yields $Q_H \| H_H$ by Eq.~(\ref{eq: fundamental})  since $k \otimes_R H \cong Q_H$ and $R \otimes_R H \cong H_H$.  The conclusion follows from \cite[Lemma 4.4]{K2014} and $d(H_H) \leq 1$.
\end{proof}

\subsection{Hopf modules  and their Fundamental Theorem relativized} Fix a Hopf subalgebra $R \subseteq H$ and let $Q$ denote $Q^H_R$ in this subsection. The details of the right $H$-module coalgebra structure on $Q$ inherited via
the canonical epimorphism $H \rightarrow Q$ are as follows:  the coproduct is given by $\cop(\overline{h}) = \overline{h\1}
\otimes \overline{h\2}$, the counit by $\eps_Q(\overline{h}) = \eps(h)$, and
the axioms of a right $H$-module coalgebra are satisfied, $\cop(qh) = q\1h\1 \otimes q\2 h\2$, as well as $\eps_Q(qh) = \eps_Q(q) \eps(h)$ for all $q \in Q, h \in H$.  

We define Ulbrich's category  $\mathcal{M}^Q_H$ with objects $X$ such that
$X_H$ is a module, $X^Q$ is a right comodule of the coalgebra $Q$ (with coaction
$\rho: X \rightarrow X \otimes Q$, $x \mapsto x\0 \otimes x\1$) and the following axiom is satisfied ($\forall x \in X, h \in H$):
\begin{equation}
\rho(xh) = x\0 h\1 \otimes x\1 h\2
\end{equation}
The arrows in this category are right $H$-module, right $Q$-comodule homomorphisms. 
Call $X$ a (right) \textit{$Q$-relative Hopf module}, since  if $R = k1_H$, then $Q = H$ and $X$ is a (right) Hopf module over $H$. 

 Given an object $X$ in this category, the $Q$-coinvariants
are given by \begin{equation}
X^{\mathrm{co}\, Q} = \{ x \in X \| \, \rho(x) = x \otimes \overline{1_H} \}
\end{equation}
Note that $X^{\mathrm{co}\, Q}$ is a right $R$-module, since $\overline{r} = \eps(r) \overline{1_H}$ for each $r \in R$. 

\begin{example}
\begin{rm}
An induced module $W \otimes_R H$ starting with an $R$-module $W_R$ is
naturally an object in $\mathcal{M}^Q_H$. The $H$-module
is given by $(w \otimes_R h)h' = w \otimes_R hh'$ and the coaction by
$w \otimes_R h \mapsto w \otimes_R h\1 \otimes \overline{h\2} $, which is well-defined since $\overline{rh} = \eps(r) \overline{h}$ for all $r \in R, h \in H$. 
\end{rm}
\end{example}

The following is a Fundamental Theorem of $Q$-relative Hopf modules, which is a clarification of  \cite[Theorem 1.3]{U} with a simplified proof. It is also a descent theorem in that it shows how to display any $Q$-relative Hopf modules as an induced $R$-module. 

\begin{theorem}
\label{th-ulbrich}
A $Q$-relative Hopf module  $V$ is an induced module of $V^{\mathrm{co}\, Q}$
in the following way: $V^{\mathrm{co}\, Q} \otimes_R H 
 \stackrel{\cong}{\longrightarrow} V$ via $v \otimes_R h \mapsto vh$.
\end{theorem}
\begin{proof}
The proof is given diagramatically in \cite{U}, but it may be noted that
an inverse mapping $V \rightarrow V^{\mathrm{co}\, Q} \otimes_R H$ is
given by $v \mapsto v\0 S(v\1) \otimes_R v\2$, where the coaction is given by
$v \mapsto v\0 \otimes \overline{v\1}$ exploiting $Q = H/R^+H$ with a choice of representative in $H$.  Since $S(r\1h\1) \otimes_R r\2 h\2 = 0$ for $r \in R^+$, this mapping is well-defined with respect to choice of representative.  Since $v\0 S(v\1) \in V^{\mathrm{co}\, Q}$ is a computation as in \cite[p.\ 569]{SY}, again correct regardless of choice of representative due to  $R^+\overline{1} = 0$, the inverse mapping is well-defined. Of course the mapping 
is checked to be an inverse just like in the proof of the Fundamental Theorem of Hopf modules. 
\end{proof}

\subsection{Existence of right $H/R$-integrals}  Given a Hopf subalgebra pair $R \subseteq H$, the paper \cite[p.\ 4]{BFG} defines a \textit{right $H/R$-integral} $t \in H$ as satisfying $th = t\eps(h) + R^+H$ for
every $h \in H$.  The existence of such a nonzero element is of course equivalent to the existence
of a nonzero $\overline{t} \in Q$ satisfying $\overline{t}h = \overline{t}\eps(h)$, also called an integral in $Q$.  Imposing
finite-dimensionality on $H$, we recall theorems in \cite{K2016, HKL} rewritten with this terminology.
\begin{theorem}\cite[''Relative Maschke Theorem'' 3.7]{K2016}
The Hopf subalgebra pair $R \subseteq H$ is a right (or left) semisimple extension $\Leftrightarrow$ $k_H \| Q_H$ $\Leftrightarrow$ $k_H$ is $R$-relative projective $\Leftrightarrow$ there is  a right $H/R$-integral $t \in H$ such that $\eps(t) = 1$ $\Leftrightarrow$ $H$ is a separable extension of $R$. 
\end{theorem}
The article \cite[Corollary 3.8]{K2016} goes on to show that $H$ is an ordinary Frobenius extension of $R$, the Nakayama automorphism and modular function of $H$ restricts to the Nakayama automorphism and modular function of $R$, respectively.  
The theorem above does not deal with a general nonzero $H/R$-integral $t$ where $\eps(t) = 0$. The paper
\cite{BFG} suggests the next two examples.
\begin{example}
\begin{rm}
Let $H \leq G$ be a group-subgroup pair, $k$ any field, and $g_1,\ldots,g_n$ a full set of right coset representative of $H$ in $G$.  Then $t= \sum_{i =1}^n g_i$ is a right $kG/kH$-integral.  Proof:
given $g \in G$, $g_i g = h_i g_{\pi(i)}$ for some $h_i \in H$ and permutation $\pi \in S_n$.  Then
$$tg - t = \sum_i h_i g_{\pi(i)} - \sum_j g_j = \sum_i (h_i - 1) g_i \in [kH]^+ kG.$$
It follows that $ta = t\eps(a)$ for the image $t \in Q$ and $a \in kG$.  Note too that with integral 
$t_H$ and $t_G$ defined as the sum of all groups elements in their respective groups, then
$t_G = t_H t$.  If the characteristic of $k$ divides any of $|H|, |G|, n$, then $\eps(t_H), \eps(t_G), \eps(t)$ is $0$ respectively. 
\end{rm}
\end{example}  
\begin{example}
\begin{rm}
Suppose $R^+H = HR^+$, i.e., $R$ is a normal Hopf subalgebra of $H$, so that $Q$ is the quotient Hopf algebra of $H \supseteq R$.  By the Larson-Sweedler theorem for finite-dimensional Hopf algebras, there is 
a nonzero right integral $\overline{t} \in Q$, then its preimage $t \in H$ is a nonzero right $H/R$-integral.  
\end{rm}
\end{example}
The relative Maschke theorem, the two examples and a third nonexample using the Taft Hopf algebra (\cite[Example 7.12]{NEFE}, \cite[Example 5.6]{K2014} and a short computation using the maximal group algebra within) suggest the following theorem. Recall
that a Hopf algebra $H$ is a $\beta$-Frobenius extension of a Hopf subalgebra $R$, where
$\beta$ is an automorphism of $R$ depending on a difference in Nakayama automorphisms of 
$R$ and $H$, or a difference in modular functions for $R$ and $H$.  In fact,  the modular function $m_H$ of $H$ restricts to the modular function $m_R$  of $R$ precisely when $H$ is an (ordinary) Frobenius extension of $R$: for textbook details on this result by Schneider \textit{et al}, see
 \cite{NEFE, SY}. 
\begin{theorem}
\label{theorem-answer}
Suppose $R \subseteq H$ is a Hopf subalgebra in a finite-dimensional Hopf algebra over a field $k$.
Then there is a nonzero right $H/R$-integral $t \in Q$ if and only if $H$ is a Frobenius extension of $R$
($\Leftrightarrow m_H |_R = m_R$).  
\end{theorem}
\begin{proof}
($\Rightarrow$) Suppose there is a nonzero right integral $\overline{t} \in Q$.  Let $t_R$ be a nonzero right integral in $R$, then $Q \stackrel{\cong}{\rightarrow} t_RH$ as $H$-modules via $q \mapsto t_R q$, since
$t_R R^+H = 0$ \cite[Lemma 3.2]{K2014}. It follows that $t_R t$ is a nonzero right integral in $H$
denoted by $t_H$. Recall that the right modular function $m_H: H \rightarrow k$ is defined by $ht_H = m_H(h)t_H$, with $m_R$ having a similar definition on $R$. Then given $r \in R$,
$$ rt_H = m_H(r)t_H = r t_R t = m_R(r)t_H$$
with the result that $m_H(r) = m_R(r) $ for all $r \in R$.  Therefore $H \supseteq R$ is a Frobenius extension.

($\Leftarrow$) Let $E: {}_RH_R \rightarrow {}_RR_R$ be a Frobenius homomorphism with dual
bases $\{ x_i \}, \{ y_i \}$ ($i = 1,\ldots,n$) \cite{NEFE}. We claim that the element in $Q$,
\begin{equation}
\label{eq: elt}
\overline{t} = \sum_{i=1}^n \eps(x_i) \overline{y_i}
\end{equation}
is a nonzero right integral in $Q$.  Note that the element $\sum_{i=1}^n x_i \otimes_R y_i$
is in $(H \otimes_R H)^H$ and the mapping $H \otimes_R H \rightarrow Q$
defined by $x \otimes_R y \mapsto \eps(x) \overline{y}$ is well-defined, since
$\eps(xr) = \eps(x) \eps(r)$ and $\overline{ry} = \eps(r) \overline{y}$ for each
 $x,y \in H, r \in R$.  It follows that $\overline{t}h = \eps(h) \overline{t}$ for all $h \in H$.  

Note that $E(R^+H) \subseteq R^+ \subseteq R^+H$, so that $E$ induces $ Q \rightarrow k \overline{1_H}$ via $\overline{h} \mapsto \overline{E(h)} = \eps(E(h)) \overline{1}$. 
From the dual bases equation $\id_H = \sum_{i=1}^n E(-x_i)y_i$, we obtain $\id_Q = \sum_{i=1}^n
\eps(E(-x_i)) \overline{y_i}$.  It follows that $\{ \overline{y_1},\ldots,\overline{y_n} \}$ is a basis of $Q$.  
If $\sum_{i=1}^n \eps(x_i) \overline{y_i} = \overline{0}$, then $\eps(x_i) = 0$ for each $i = 1,\ldots,n$.  Then $x_1,\ldots,x_n \in H^+$, which contradicts the dual bases equation
$\sum_{i=1}^n x_i E(y_i -) = \id_H$.  
\end{proof}
For example, a Hopf algebra $H$ within its (always unimodular) Drinfeld double $D(H)$ is a Frobenius extension if and only if $H$ is unimodular.  In general, a Frobenius coordinate system for $H \| R$ in terms of a nonzero integral $t \in Q$
is given by dual bases tensor $S(t\1) \otimes_R t\2$ and Frobenius homomorphism
$E: H \rightarrow R$ given by $E(h) = t^* \rightharpoonup h$, where $t^* \in Q^*$
satisfies $t^*(t) = 1$, $q^* t^* = q^*(\overline{1_H}) t^*$ for every $q^* \in Q^*$ and $rt^* = \eps(r) t^*$ for every $r \in R$.  
\subsection{Short exact sequence of quotient modules for a tower}
Let $K \subseteq R \subseteq H$ be a tower of Hopf subalgebras in a finite-dimensional Hopf algebra $H$.  Note the transitivity lemma.
\begin{lemma}
The quotient modules of the tower $K \subseteq R \subseteq H$ satisfy
$$Q^H_K \cong Q^R_K \otimes_R H$$ as right $H$-modules.
\end{lemma}
\begin{proof}
This follows from $Q^H_K \cong k \otimes_K H$ (\cite{K2014}), similarly 
$Q^R_K \cong k \otimes_K R$ and the cancellation, 
$R \otimes_R H \cong H$.  
\end{proof}
Since $K^+R \subseteq K^+H \subseteq R^+H$,
we note a short exact sequence,
\begin{equation}
\label{eq: ses1}
 0 \rightarrow R^+H/K^+H \rightarrow H/K^+H \rightarrow H/R^+H \rightarrow 0. 
\end{equation}
Denote the kernel of the counit on $Q^R_K$ by ${}^+Q^R_K$.  
\begin{prop}
The quotient modules of the tower $K \subseteq R \subseteq H$ satisfy
\begin{equation}
\label{eq: ses2}
0 \rightarrow {}^+Q_K^R \otimes_R H \rightarrow Q^H_K \rightarrow Q^H_R \rightarrow 0,
\end{equation}
with respect to canonical mappings. 
\end{prop}
\begin{proof}
Follows from the short exact sequence~(\ref{eq: ses1}) and the isomorphism
$${}^+Q_K^R \otimes_R H \stackrel{\sim}{\longrightarrow} R^+H/K^+H$$
given by $(r + K^+R) \otimes_R h \mapsto rh + K^+H$ where $r \in R^+$, since the mapping
is surjective between $k$-spaces of equal dimension. 
Also follows from the lemma above and tensoring the short exact sequence,
$$ 0 \rightarrow {}^+Q^R_K \rightarrow Q^R_K \rightarrow k \rightarrow 0$$
by the exact functor $- \otimes_R H$ (as ${}_RH$ is a free module). 
\end{proof}
\begin{cor}
If $R^+H \into H$ is split as right $H$-modules, then $d(Q^H_K)$ and 
 $d(Q^H_R)$ are both finite as right $H$-modules.  
\end{cor}
\begin{proof}
Since $R^+H$ is a projective(-injective) $H$-module,
then so is $Q^H_R$ from the short exact sequence,
\begin{equation}
\label{eq: ses3}
 0 \rightarrow R^+H \rightarrow H \rightarrow Q_R^H \rightarrow 0,
\end{equation}
which of course splits. 
Then $R$ is  a semisimple Hopf algebra by \cite[Theorem 3.5]{K2014}.  Then the Hopf subalgebra $K$ is semisimple \cite[3.2.3]{M}. Then 
$Q^H_K$ is projective.  But projective modules in mod-$H$ have finite depth
\cite[Prop.\ 4.5]{K2014}. 
\end{proof}
This proof demonstrates that  a fourth equivalent condition one may add to
\cite[Theorem 3.5]{K2014}, which characterizes the semisimplicity of $R$, is that $R^+H$ is a projective $H$-module (and see the sufficient condition below in Prop.~\ref{prop-gen}). 

\subsection{Mackey Theory for Quotients of Group Algebras}
\label{subsection-Mackey}
We change notation from Hopf to group notation in this subsection. We review some Mackey theory in a special context relevant to establishing an upper bound on h-depth in terms of the number of conjugates intersecting in the core.   Let $G$ be a  finite group and
$H,K \leq G$ be two subgroups, $Q^G_K$ the quotient module $k$-coalgebra as in Example~\ref{example-group}, and $Q^G_K \downarrow_H$ the restriction of $Q^G_K$ from $G$-module to $H$-module.  Recall that $Q^G_K \cong k \otimes_{kK} kG$ which is the induced module
denoted by $k \uparrow^G$. 

If $N$ is an arbitrary $K$-module, and $g_i \in K \setminus G / H$ is a set of double coset representatives of $K,H$ in $G$, and $K^{g}$ denotes the conjugate subgroup $g^{-1}K g$ for $g \in G$,  Mackey's formula for the induced $G$-module of $N$ restricted to $H$ is given by

 \begin{equation}
\label{eq: Mackey}
N\uparrow^G \downarrow_H \cong \sum_{g_i \in  K \setminus G / H} \oplus \, N\! \otimes_K\! g_i \downarrow_{K^{g_i} \cap H} \uparrow^H
\end{equation}
It follows from an application to $N = k$ that
\begin{equation}
\label{eq: Quacky}
Q_K^G \downarrow_H \cong \sum_{g_i \in  K \setminus G / H}\! \oplus \,  Q^H_{K^{g_i} \cap H}
\end{equation}
By Eq.~(\ref{eq: fundamental}), we note that $Q^G_K \otimes Q^G_H \cong Q^G_K \otimes_{kH} kG$ as $G$-modules,
then applying Eq.~(\ref{eq: Quacky}) and the transitivity lemma obtains
\begin{equation}
\label{eq: Quacky2}
Q^G_K \otimes Q^G_H \cong \sum_{g_i \in  K \setminus G / H}\! \oplus \,  Q^G_{K^{g_i} \cap H} 
\end{equation}
 It follows from induction (alternatively, the Mackey Tensor Product Theorem) that the tensor powers of $Q^G_H$ in mod-$G$ are given by 
\begin{equation}
\label{eq: tensorpowers}
(Q^G_H)^{\otimes( n+1)} \cong \sum^{}_{g_{i_1},\ldots,g_{i_n} \in  H \setminus G / H } \!\oplus \, Q^G_{H^{g_{i_1}} \cap \cdots \cap H^{g_{i_n}} \cap H} 
\end{equation}
(in all $|H: G : H |^{n-1}$ nonunique $Q$-summands). 
Recall that the core of a subgroup $H \leq G$ is the largest normal subgroup of $G$ within $H$. Also core$_H(G) = \cap_{g \in G} H^g$.

\begin{theorem}
\label{theorem-sep}
Suppose $R = kH$ is a separable $k$-algebra.   \newline If  there are elements of $G$ such that 
core$_H(G) = H^{g'_1} \cap \cdots \cap H^{g'_r} \cap H $,  \newline 
then $d(Q_H^G) \leq r+1$.
\end{theorem}
\begin{proof}
Let $Q = Q^G_H$. 
It suffices to prove the similarity, $Q^{\otimes (r+1)} \sim Q^{\otimes (r+2)}$. The tensor powers of
$Q$ are given by Eq.~(\ref{eq: tensorpowers}).  Since $Q^{\otimes (r+1)} \| Q^{\otimes (r+2)}$, 
we are left with showing that an arbitrary $Q$-summand 
$Q^G_{H^{g_1} \cap \cdots \cap H^{g_{r+1}} \cap H}$ in $Q^{\otimes (r+2)}$ divides a multiple of
a $Q$-summand in $Q^{\otimes (r+1)}$, which we can take to be $Q^G_{\mbox{core}_H(G)}$ from
the hypothesis on the core.  
The proof follows from applying the short exact sequence~(\ref{eq: ses2}) to the group algebras of
the tower core$_H(G) \leq H^{g_1} \cap \cdots \cap H^{g_{r+1}} \cap H \leq G$.  Since the characteristic of $k$ does not divide the order of $H$ (by hypothesis), it does not divide the order
of core$_H(G)$ or any other subgroup of $H$.  Thus their group algebras are semisimple.  It
follows that the leftmost module of the short exact sequence~(\ref{eq: ses2}) we are considering
is projective-injective, whence the sequence splits.  That is $Q^G_{H^{g_1} \cap \cdots \cap H^{g_{r+1}} \cap H} \| Q^G_{\mbox{core}_H(G)}$ indeed.   
\end{proof}
Recall from Section~1 that for a Hopf algebra-Hopf subalgebra pair $R' \subseteq H'$, the
h-depth satisfies $d_h(R',H') = 2 d(Q^{H'}_{R'}) +1$.  From this follows the corollary. 
\begin{cor}
\label{cor-h}
Given a subgroup and ground field under the hypotheses of Theorem~\ref{theorem-sep}, the h-depth satisfies  $d_h(kH, kG) \leq 2r+3$.
\end{cor}
Combinatorial depth is first defined in \cite{BDK}.  A certain simplification in the definition of minimum even combinatorial depth of a subgroup pair $H \leq G$, denoted by $d_c^{ev}(H,G)$,were highlighted in \cite{HHP2}as follows.  Let $\mathcal{F}_0 = \{ H \}$ and for each $i \in \N$,  $$\mathcal{F}_i = \{ H \cap H^{x_1} \cap \cdots \cap H^{x_i} \| \, x_1,\ldots,x_i \in G \}. $$
Note that $\mathcal{F}_0 \subseteq \mathcal{F}_1 \subseteq \mathcal{F}_2 \subseteq \cdots$. 
If the sequence of subsets ascends strictly until $\mathcal{F}_{n-1} = \mathcal{F}_n$, then
$d_c^{ev}(H,G) = 2n$.  Minimum combinatorial depth $d_c(H,G)$ satisfies $d^{ev}_c(H,G) -1 \leq d_c(H,G) \leq d^{ev}_c(H,G)$; the precise determination is explained in \cite{BDK, HHP, HHP2}. A particularly easy characterization is
$d_c(H,G) = 1$ if and only if $G = H C_G(H)$ \cite{BDK}.  
\begin{prop}
 Under the hypotheses on the subgroup pair $H \leq G$ and the ground field $k$ in Theorem~\ref{theorem-sep}, h-depth and combinatorial depth satisfy
$$d_h(H,G) \leq d_c(H,G) + 2.$$
\end{prop}
\begin{proof}
Suppose $d^{ev}_c(H,G) = 2n$.  Then $\mathcal{F}_{n-1} = \mathcal{F}_n$.  A look at formula~(\ref{eq: tensorpowers}) for the tensor powers of $Q^G_H := Q$ reveals that $Q^{\otimes n} \sim Q^{\otimes (n+1)}$.  Hence $d(Q_G) \leq n$, and so $d_h(H,G) \leq 2n+1$, i.e.,
$d_h(H,G) \leq d^{ev}_c(H,G) + 1 \leq d_c(H,G) + 2$.  
\end{proof}
\begin{example}
\begin{rm}
Suppose $H \triangleleft G$.  It follows we may apply the theorem and corollary with $r = 0$.
Then $d_h(kH,kG) \leq 3$. Also combinatorial depth is $d_c(H,G) \leq 2$ \cite{BDK} since $\mathcal{F}_0 = \mathcal{F}_1$. 

Next consider the permutation groups $S_n \leq S_{n+1}$. It is an exercise that \newline   core$_{S_n}(S_{n+1}) = \{ (1) \}$ and it takes only $r = n-1$ conjugate subgroups of $S_n$ to intersect trivially \cite{BKK}. By Corollary~\ref{cor-h} $d_h(\C S_n, \C S_{n+1}) \leq 2n+1$.
In fact, $d_h(\C S_n, \C S_{n+1}) = 2n+1$ by \cite[Lemma 5.4]{K2014}.  Also $d_c(S_n,S_{n+1}) = 2n-1$ \cite{BDK}.  It follows that the inequality in the proposition cannot be improved in general.  
\end{rm}
\end{example}
\begin{example}
\begin{rm}
Suppose $H < G$ is a non-normal trivial-intersection (TI) subgroup of a finite group; i.e., $H \cap gHg^{-1} = E := \{ 1_G \}$ for every $g \in G - N_G(H) \neq \emptyset$.  It follows that $\mathcal{F}_1 = \{ H, E \} = \mathcal{F}_2$, so
$d^{ev}_c(H,G) = 4$ (\cite{BDK} shows $d_c(H,G) = 3$).  It may also be computed easily 
that Eq.~(\ref{eq: Quacky}) reduces to $Q_H \cong m_1 \cdot k_{\eps} \oplus m_2 \cdot kH$ for some $m_i \in \N$, and that Eq.~(\ref{eq: Quacky2}) reduces to $Q^{\otimes 2}_H \cong s_1 \cdot Q_H \oplus s_2 \cdot kH$ for some $s_i \in \N$ ($i=1,2$).
We conclude that $d(Q_H) = 1$ and $d_{ev}(kH,kG) = 4$.  Moreover, $Q^{\otimes 2} \cong n_1 \cdot Q \oplus n_2 \cdot kG$, for some $n_i \in \N$, from which we conclude that $d(Q_G) = 2$ and  $d_h(kH,kG) = 5$.  
\end{rm}
\end{example}

\subsection{A Mackey result generalized to certain Hopf algebras}  The following is exercise 5.2 in \cite{I}:
let $H,K \leq G$ be subgroups of a finite group such that $HK = G$. Suppose $\psi$ is a class function of $H$. Use Mackey's Theorem to show that $\psi \uparrow^G\downarrow_K = \psi \downarrow_{H \cap K}\uparrow^K$.  

Note that the character of $Q^G_H$ is $ \eps \uparrow^G$ where $\eps$ is the counit on $kH$, equivalently, the principal character of $H$.  
In this case the following proposition somewhat generalizes the exercise  for certain  Hopf subalgebras of a finite-dimensional Hopf algebra without recourse to a Hopf algebra version of Mackey theorem.

We say that a Hopf algebra $H$ has linear disjoint Hopf subalgebras 
$R,K$  if $H = RK$ and the multiplication epimorphism  
$R \otimes_{R \cap K} K \rightarrow H$ is an isomorphism; equivalently, $H = RK$ and 
\begin{equation}
\label{eq: dims}
\dim H = \frac{\dim R \dim K}{\dim R \cap K}.
\end{equation}  
 See the example below in this subsection. 
Note that any two Hopf subalgebras of a finite group algebra
$kG$ are linear disjoint, since a Hopf subalgebra
is necessarily the group algebra of a subgroup, and a lemma holds for order
of two subgroups and their join corresponding to the dimension equation~(\ref{eq: dims}): prove it with the orbit counting theorem or see \cite{A}. Also a commutative
Hopf algebra has linear disjoint Hopf subalgebras by  \cite[Prop.\ 6]{T}. 
\begin{prop}
Suppose $R, K \subseteq H$ are linear disjoint Hopf subalgebras of a finite-dimensional Hopf algebra $H$, where
$RK = H$ and $B$ denotes the Hopf subalgebra $R \cap K$.  Then $Q^H_R \cong Q^K_B$ as $K$-modules.  In this case, for every $H$-module $M$, there is an isomorphism of $K$-modules $M\otimes_R H \cong M \otimes_B K$.  
\end{prop}
\begin{proof}
Since $B^+K \subseteq R^+ K$ and $R^+H = R^+ RK = R^+ K$, we may map 
$Q^K_B \rightarrow Q^H_R$ as $K$-modules by $x + B^+K := \tilde{x} \mapsto x + R^+H$ for $x \in K$. Also
$\overline{rx} = \eps(r) \overline{x}$ in $RK / R^+K$, so this mapping is onto. 

 It is injective, since $$\dim Q^H_R = \frac{\dim H}{\dim R} = \frac{\dim K}{\dim B} = \dim Q^K_B.$$

The last statement follows from $M \otimes Q^H_R \cong M \otimes Q^K_B$
as $K$-modules, and two applications of Eq.~(\ref{eq: fundamental}). 
\end{proof}

\begin{example}
\label{example-smallqg}
\begin{rm}
The following example illustrates the  proposition: let $H$ be the  small quantum group $ \overline{U_q}(sl_2)$ of dimension $n^3$ with the usual generators $K,E,F$, 
with $q$  a primitive $n$'th root of unity in $k = \C$, and  $n$  odd, 
where  $K^n = 1$, $E^n = 0 = F^n$, $EF - FE = \frac{K  - K^{-1}}{q - q^{-1}}$, $KE = q^2EK$, and $KF = q^{-2}FK$.  This is a $n^3$-dimensional Hopf algebra with coproduct given by $\cop(K) = K \otimes K$, $\cop(E) = E \otimes 1 + K \otimes E$ and $\cop(F) = F \otimes K^{-1} + 1 \otimes F$. The counit satisfies $\eps(K) = 1$, $\eps(E) = 0 = \eps(F)$. Also the antipode values are determined as an exercise.  

Let  $R_1$  the Hopf subalgebra of dimension $n^2$ generated by $K,F$ and $R_2$ the Hopf subalgebra of dimension $n^2$ generated by $K, E$. Both Hopf subalgebras are isomorphic to the Taft algebra of same dimension.   Note that $B$ is the cyclic group algebra of dimension $n$ generated by $K$.
\end{rm}
\end{example}

%%%%%%%%%%%%%%%%%%%%%%%%%%%%%%%%%%%%%%%%%%%%%%

\section{Core Hopf ideals of  Hopf subalgebras}
Let $R$ be a Hopf subalgebra in a finite-dimensional Hopf algebra $H$.  Let $Q = Q^H_R$ be the right quotient module coalgebra of $R \subseteq H$, as defined above.  We review what we know about
the chain of annihilator ideals of the tensor powers of $Q$ in mod-$H$ \cite{HKY, HKS, K2016}.  
First it is worth noting that $Q$ is cyclic module equal to $H/R^+H$, where of course $R^+H$ is a 
right ideal.  The ring-theoretic \textit{core} is the largest two-sided ideal within $R^+H$, which an exercise will reveal to be $\Ann Q_H$.  This notion is mentioned in \cite[p.\ 54]{Lam}, also noting in
\cite[11.5]{Lam3} that $Q_H$ is faithful (i.e. $\Ann Q_H =  0 $) if and if $R^+H \cap Z(H) = 0$, where
the center of $H$ is denoted by $Z(H)$.

\begin{prop}
\label{prop-gen}
If $Q$ is a generator $H$-module, then $R$ is semisimple.
\end{prop}
\begin{proof}
Recall that a module over a QF-algebra is a generator if and only if it is faithful.  If $R$ is not semisimple, it has a nonzero left integral $\ell_R$ in $R^+$.  One may show as an exercise using
the freeness of ${}_RH$ and one-dimensionality of the space of integrals that there is $\Lambda \in H$
such that a nonzero left integral in $H$, $\ell_H = \ell_R \Lambda$.  Then the one-dimensional ideal
spanned by $\ell_H \in R^+H$.  It follows that the core is nonzero, and therefore $\Ann Q_H \neq 0$, i.e., $Q_H$ is not faithful.
\end{proof}

The following is a descending chain of two-sided ideals in $H$:
\begin{equation}
\label{eq: dca}
\Ann Q \supseteq \Ann (Q_. \otimes Q_.) \supset \cdots \supset \Ann Q^{\otimes n} \supseteq \cdots
\end{equation}
It follows from Rieffel's classical theory (for any $H$-module $Q$) extended by Passman-Quinn and Feldv\"oss-Klingler, that the chain stabilizes at some $n$ denoted by $\ell_Q$, that $\ell_Q$ is the least $n$ for which $\Ann Q^{\otimes n}$ is a Hopf ideal $I$ in the ring-theoretic core $\Ann Q_H$.  In fact, $\Ann Q^{\otimes \ell_Q} = I$ is the maximal Hopf ideal in $\Ann Q_H$, called the \textit{Hopf core ideal} of $R \subseteq H$ \cite[Section 3.4]{K2016}.  Since (\ref{eq: dca}) is a chain of $H^e$-modules, $\ell_Q$ is bounded above by the Jordan-H\"older length of $H$ as an $H^e$-module.  

Now if two modules in mod-$H$ are similar, such as $Q^{\otimes n} \sim Q^{\otimes (n+1)}$,
it is easy to see that their annihilator ideals are equal in $H$. It follows that
the length $\ell_Q$ and depth $d(Q_H)$ satisfy the inequality,
\begin{equation}
\ell_Q \leq d(Q_H)
\end{equation}
  If $H$ is a semisimple algebra, the converse holds:  two modules in mod-$H$ are similar if they have equal annihilator ideals in $H$.  
This follows from noting that the $2^n$ ideals in $H$, where $n$ is the number of blocks, are described by annihilator ideals of direct sums of simples \cite{K2016}.  In this way the following theorem is deduced.  
\begin{theorem}\cite[Theorem 3.14]{K2016}
\label{theorem-ell}
If $H$ is semisimple and $R$ is a Hopf subalgebra, then h-depth satisfies $d_h(R,H) = 2\ell_Q + 1$. 
\end{theorem}
\begin{proof}
These is an alternative proof offered below in terms of the Dress category and as a corollary of Proposition~\ref{prop-sigmacat}.
\end{proof}
\begin{example}
\begin{rm}
If $R \neq H$ is a normal Hopf subalgebra of a semisimple Hopf algebra, then $R^+H$ is a Hopf ideal, $I = R^+H$, so $\ell_Q = 1$.
It follows that $d_h(R,H) = 3$, as in Example~\ref{example-qha}. The author is unaware
of any examples of non-normal Hopf subalgebras or even non-normal subgroups that have h-depth 3. 

If $R \subseteq H$ is a semisimple Hopf subalgebra pair with $Q_H$ faithful, or equivalently a generator $H$-module, then $Q \sim Q^{\otimes 2}$.  Again $d_h(R,H) = 3$ if $R \neq H$. The situation is more complicated for nonsemisimple Hopf algebras, since nonprojective indecomposables must be taken into account before concluding that tensor powers of $Q$ are similar; see for example \cite{HKL}.  
  
Suppose $K \subseteq G$ is a subgroup in a finite group, and consider the groups algebras over  any field $k$.  Then the right quotient module $Q$ is the $k$-coalgebra on the set of right cosets of $K$ in $G$.  The length $\ell_Q$ of the descending chain of annihilator ideals of increasing tensor powers of Q then satisfies
$d_h(K,G) = 2 \ell_Q + 1$ where the field is understood. An exercise, which uses  the (Passman-Quinn) fact that Hopf ideals  in $kG$ correspond to normal subgroups in $G$, shows that the maximal
Hopf ideal in $\Ann Q_G$ is $kN^+ kG$, where  $N := $core$_K(G)$ \cite[Theorem 3.13]{K2016}.

\end{rm}
\end{example}
We recall the definition in \cite{B} of \textit{core Hopf subalgebra} of a Hopf subalgebra pair $R \subseteq H$: it is the maximal normal Hopf subalgebra in $H$ that is contained in $R$.  
\begin{prop}
Let $I$ be the maximal Hopf ideal in $R^+H$.  If $K$ is the maximal normal Hopf subalgebra in $R$, then $HK^+ \subseteq I$.  Conversely, if $I = HK^+$ for some Hopf subalgebra $K \subseteq R \subseteq H$, then
$K$ is normal in $H$.
\end{prop}
\begin{proof}
(First statement) If $\overline{h} \in Q^H_R$ and $x \in HK^+ = K^+H$, then $\overline{h}x
= \overline{hx} = 0$ since $hx \in K^+H \subseteq R^+H$. Then $HK^+$ is a Hopf ideal in $\Ann Q_H$, whence by maximality of $I$ we obtain $HK^+ \subseteq I$.  

(Second statement) Since $I$ is a Hopf ideal, it is invariant under the antipode, so
$HK^+ = I = S(I) = K^+H$, i.e., $K$ is normal in $H$.  
\end{proof}
As an application of the descending chain of annihilator ideals in Eq.~(\ref{eq: dca}), note the following proposition, which  comes tantalizing close to solving the finiteness question for depth of Hopf subalgebras \cite[Problem 1.1 or 1.2]{K2016}. Given a finite-dimensional algebra $C$ and module $M_C$, consider the subcategory subgenerated by $M$ in mod-$C$  denoted by $\sigma[M]$; i.e., $\sigma[M]$ has objects that are submodules of quotients of finite direct sums of $M$ (cf.\ \cite[18F]{Lam2} and \cite{BW}).  For example, in mod-$H$ we have $$\sigma[Q] \subset \sigma[Q \otimes Q] \subset \cdots \subset \sigma[Q^{\otimes \ell_Q}] \subseteq \cdots,$$
since $Q^{\otimes n} \| Q^{\otimes (n+1)}$ in mod-$H$.  The proposition below notes that the ascending chain of subcategories will stop growing at the tensor power $\ell_Q$.
\begin{prop}
\label{prop-sigmacat}
Given a Hopf subalgebra pair $R \subseteq H$ with quotient $Q$ and natural number $\ell_Q$ as above, the subcategories $\sigma[Q^{\otimes \ell_Q}] = \sigma[Q^{\otimes (\ell_Q + n)}]$ for all $n \in \N$. 
\end{prop}
\begin{proof}
This follows from Eq.~(\ref{eq: dca}) and the following fundamental fact about $\sigma[M]$, that if $I = \Ann M_C$, then $\sigma[M] = $ mod-$C/I$ (cf.\ \cite[p.\ 505, ex.\ 33]{Lam2}.  For example,  if $Q$ is faithful, then it is a generator in mod-$H$, so that $\sigma[Q] =$ mod-$H$.  Denote $\Ann Q^{\otimes \ell_Q} = I$, a Hopf ideal in $\Ann Q_H$ and equal to the annihilator ideal of all higher tensor powers of $Q$.  Then $\sigma[Q^{\otimes (\ell_Q+n)}] = $mod-$H/I$. 
\end{proof}
%%%%%%%%%%%%%%%%%%%%%%%%%%%%%%%%%%%%%%%%%%%%%%%%%

\section{The Endomorphism Algebra of $Q$ and its Tensor Powers}
\label{section-endoQ}

By the Dress category Add$[M_C]$, we mean the subcategory of $C$-modules isomorphic to direct summands  of $M \oplus \cdots \oplus M$ for any multiple of $M$.  If $I_1,\ldots,I_m$ are the pairwise nonisomorphic indecomposable constituents of $M$ in a Krull-Schmidt decomposition, then
Add$[M] =$ Add$[I_1 \oplus \cdots \oplus I_m]$. Moreover, Add$[M]$ = Add$[N]$ if and only if $M \sim N$, i.e., the two $C$-modules are similar.  
 If $C$ is a semisimple algebra,
Add$[M]$ = $\sigma[M]$, since all monics and epis split.  It follows from
Proposition~\ref{prop-sigmacat} that $$Q^{\otimes (\ell_Q + 1)} \| m \cdot Q^{\otimes \ell_Q}$$
i.e., $d(Q_H) = \ell_Q$,  a somewhat different proof of Theorem~\ref{theorem-ell}.  

Consider now the endomorphism algebra $E := \End M_C$, and the subcategory of mod-$E$ whose objects are projective modules, denoted by $\mathcal{P}(E)$. We have the natural bimodule ${}_EM_C$ falling out from this.  
Then the categories \textrm{Add}[$M_C]$ and $\mathcal{P}(E)$ are equivalent via
functors $X_C \mapsto \Hom (M_C, X_C)$ and $P_E \mapsto P \otimes_E M_C$ as one may check. This exercise proves the well-known lemma:
\begin{lemma}
\label{lemma-addproj}
Suppose $C$ is a ring, $M_C$ a module and $E := \End M_C$. Then the category Add[$M_C$] is equivalent  to $\mathcal{P}(E)$.  
\end{lemma}

 The next proposition is important to considerations of depth of $Q$. 
\begin{prop}
\label{prop-conv}
Suppose $C$ is a finite-dimensional algebra, $M_C, N_C$ are finite-dimensional modules satisfying for some
$r \in \N$, 
$$M_C \oplus * \cong r \cdot N_C,$$
 and $E_M, E_N$ are their endomorphism algebras.  Then $M$ and $N$ are similar $C$-modules if and only if $E_M$ and $E_N$ are Morita equivalent algebras.
\end{prop}
\begin{proof}
First  note that Add$[M_C] \subseteq $ Add$[N_C]$. Both module subcategories are of course of finite representation type, since $M, N$ have only finitely many nonisomorphic indecomposable constituents, denoted by Indec$[M] = \{ I_1,\ldots,I_n\} \subseteq $ Indec$[N] = \{I_1,\ldots, I_{n+m} \}$.  

($\Rightarrow$) True without the displayed condition. 

($\Leftarrow$) It will suffice to prove that $m = 0$. Assuming that mod-$E_M$ and mod-$E_N$ are equivalent categories, we restrict the inverse functors of tensoring by Morita progenerators to obtain that  $\mathcal{P}(E_M)$
and $\mathcal{P}(E_N)$ are  themselves equivalent categories.  Then by Lemma~\ref{lemma-addproj}, 
Add$[M]$ and Add$[N]$ are equivalent categories.  Equivalences preserve indecomposable modules, so that $m = 0$.
\end{proof}
As a consequence, the problem of finite depth of a Hopf subalgebra \cite[Problem 1.1 or 1.2]{K2016} is equivalent to the following problem.
\begin{problem}
\label{prob-morita}
Are there two tensor powers of $Q_H$ with Morita equivalent endomorphism algebras?
\end{problem}
For any $m \in \N$ and Hopf subalgebra quotient module $Q$, there is
an injective homomorphism $\End Q^{\otimes m}_H \into \End Q^{\otimes (m+1)}_H$ given by $\alpha \mapsto \alpha \otimes \id_Q$. The result is a tower of endomorphism algebras of increasing tensor powers of the quotient module of a Hopf subalgebra:
\begin{equation}
\label{eq: tower}
\End Q_H \into \End (Q \otimes Q)_H \into \cdots \into \End {Q^{\otimes n}}_H \into \cdots
\end{equation}
If $H$ is semisimple as a $k$-algebra, where $k$ is an algebraically closed field of characteristic zero, $Q$ and its tensor powers are semisimple modules, and their endomorphism algebras are semisimple $k$-algebras by Schur's Lemma. It follows that the tower of algebras is split, separable and Frobenius at every monic or composition of monics (using the construction of a very faithful conditional expectation in \cite{GHJ}, as noted in \cite[p.\ 30]{BK}). 

If $H$ is a left or right semisimple extension of $R$, hence separable Frobenius \cite{K2016}, then the exact sequence of right $H$-modules
\begin{equation}
\label{eq: seq}
0 \longrightarrow Q^+ \longrightarrow Q \stackrel{\eps_Q}{\longrightarrow} k \longrightarrow 0
\end{equation}
splits \cite[Theorem 3.7]{K2016}. Then tensoring in mod-$H$ by $Q^{\otimes n}$ from the left yields a split exact sequence of $H$-modules, 
\begin{equation}
\label{eq: derivedseq}
0 \longrightarrow Q^{\otimes n} \otimes Q^+ \longrightarrow Q^{\otimes (n+1)} \longrightarrow Q^{\otimes n} \longrightarrow 0.
\end{equation}
The resulting $H$-module decomposition of $Q^{\otimes (n + 1)} \cong Q^{\otimes n} \oplus Q^{\otimes n} \otimes Q^+$ results in an expression of $\End {Q^{\otimes (n+1)}}_H$ as a $2 \times 2$-matrix algebra with  the mapping $\alpha \mapsto \alpha \otimes \id_Q$ becoming diagonal, where $\alpha \in \End {Q^{\otimes n}}_H$ is in the upper lefthand corner. Then the  Tower~(\ref{eq: tower})  is composed  of split extensions \cite{NEFE}.  At $n = 0$ one sees that $k_H \| \End Q_H$, which implies that $\End Q_H$ is semisimple if the category of finite-dimensional modules over
$\End Q_H$ is a (for example) finite tensor category \cite{EO}. 

In trying to answer Problem~\ref{prob-morita} with perhaps a counterexample, it is useful to point out a well-developed theory of endomorphism algebras of tensor powers of certain modules over  groups and quantum groups;  these are related to Schur-Weyl duality, its generalizations, 
Hecke algebras, Temperley-Lieb algebras, BMW algebras and representations of braid groups; see for example \cite{SBA, BBH}.  
Of course, $Q$ can take on interesting ``values'' from this point of view; e.g., $Q$ for $\C S_{n-1} \subset \C S_n$ is the standard $n$-dimensional representation of the permutation group $S_n$ (\cite[Prop.\ 3.16]{K2016}, whose n'th tensor power is a notably different module than in Schur-Weyl theory).  In general for semisimple group-subgroup algebra pairs, the knowledge of the subgroup depth informs us at what stage in the tower of endomorphism algebras of tensor powers of $Q$ the algebras have equally many vertices in their Bratteli diagram. 

In trying to answer Problem~\ref{prob-morita} in the affirmative, it is worth noting that 1) we may assume without loss of generality that $\Ann Q_H$ does not contain a nonzero Hopf ideal (if so, mod out without a change in the h-depth); 2) At the natural number $n = \ell_Q$, $Q^{\otimes n}$ is a faithful $H$-module, therefore a right $H$-generator, and $Q^{\otimes n}$ is then a left $E_n := \End Q^{\otimes n}_H$-projective module \cite{AF}; 3) Then there is a reasonable expectation that there is $m \geq n$ such that ${}_{E_n}Q^{\otimes m} \sim {}_{E_n}Q^{\otimes (m+1)}$ given good properties of
Tower~(\ref{eq: tower}) and the following lemma. It is conceivable that under some circumstances the last point yields two tensor powers of $Q$ that are similar as $H$-modules, which answers the problem in the affirmative (for whatever hypotheses are introduced).  

\begin{lemma}
\label{lemma-splitepibimods}
For any $n \in \N$, ${}_{E_n}Q^{\otimes (n+1)}_H \oplus * \cong {}_{E_n}Q^{\otimes (n+2)}_H$.  Consequently, $Q^{\otimes (n+1)} \| Q^{\otimes (n+m)}$ as $E_n$-$H$-bimodules for any $m \geq n+1 \geq 2$.  
If $H$ is a semisimple extension of $R$, this may be improved to $Q^{\otimes n} \| Q^{\otimes (n+m)}$ as $E_n$-$H$-bimodules for any $m \geq n \geq 0$. 
\end{lemma}
\begin{proof}
The inclusion we are working with is $E_n \into E_{n+1}$, $\alpha \mapsto \alpha \otimes \id_Q$ for all $n$ as above.  The mapping $Q^{\otimes (n+2)} \rightarrow Q^{\otimes (n+1)}$ given by $$q^1 \otimes \cdots \otimes q^{n+2} \longmapsto q^1 \otimes \cdots \otimes q^{n+1} \eps_Q(q^{n+2})$$ is an $E_n$-$H$-bimodule   split epimorphism with section $Q^{\otimes (n+1)} \rightarrow Q^{\otimes (n+2)}$ given by
\begin{equation}
\label{eq: interchangeable}
q^1 \otimes \cdots \otimes q^{n+1} \longmapsto q^1 \otimes \cdots \otimes q^n \otimes \cop_Q(q^{n+1}). 
\end{equation}
This establishes the first statement in the lemma.  Similar is the proof that  $${}_{E_{n+1}}Q^{\otimes (n+2)}_H \| {}_{E_{n+1}}Q^{\otimes (n+3)}_H.$$  But this restricts to $E_n$-$H$-bimodules, so
as $E_n$-$H$-bimodules $Q^{\otimes (n+1)} \| Q^{\otimes (n+3)}$.  The second statement is then proven by a straightforward induction on $m$.  The last statement is proven from the equivalent hypothesis that there is $t \in Q$ such that $th = t\eps(h)$ for every $h \in H$ and $\eps_Q(t) = 1$ \cite[Theorem 3.7]{K2016}.  Then define a new section in Eq.~(\ref{eq: interchangeable}) by
$$q^1 \otimes \cdot \otimes q^{n+1} \longmapsto q^1 \otimes \cdots \otimes q^{n+1} \otimes t, $$
which is clearly an $E_{n+1}$-$H$ morphism, like the epi above.  
\end{proof}
A basic lemma in the subject is the following, adapted to the language of this paper.
Suppose $G$ is a finite group with subgroup $K$, and $k$ is an algebraically closed field of characteristic zero. Note the
idempotent integral element $e$ in $K$ given by $e := \frac{1}{|K|} \sum_{x \in K} x$, also the  separability idempotent $(S \otimes \id)\cop(e)$ for the algebra $R$. 
\begin{lemma}
\label{lem-Hecke}
 Let $H = kG$ and $R = kK$ be the corresponding group algebras and Hopf subalgebra pair.  Then $\End Q_H \cong eHe$, the Hecke algebra of $(G,K, 1_K)$  \cite[11D]{CR}. 
\end{lemma}
\begin{proof}
First identify an arbitrary element $\sum_{g\in G} n_g g$ in $H$ with the function $G \rightarrow k$ in $k^G$ given by $g \mapsto n_g$.  
The product on $H$ is then isomorphic to the convolution product given by $f * h (y) = \sum_{x \in G} f(yx^{-1}) h(x)$ for $f,h \in k^G$.  Since $\sum_{x \in G} x^{-1} \otimes x$ is a Casimir element of $(H \otimes H)^H$, we see that characteristic functions of double cosets in $H$ in $G$ span a nonunital subalgebra of $H$ normally thought of as the Hecke algebra of a subgroup pair.  

Next recall that the quotient right $H$-module $Q \cong eH$ since we may choose $t_R = e$. By a well-known identity in ring theory \cite{Lam}, $\End eH \cong eHe$ via left multiplication and evaluation at $e$.    If $\gamma_1, \ldots,\gamma_t$ are the double coset representatives of $K$ in $G$, a computation shows that an arbitrary element of $eHe$ is identified in the Hecke algebra as follows.   
$$e(\sum_{g \in G} n_g g)e = \sum_{i =1}^t (\sum_{x \in K\gamma_i K} \frac{n_x}{|K : G : K|})
\chi_{K\gamma_i K}$$
where $\chi_X$ is the characteristic function in $k^G$ of a subset $X \subseteq G$. Let
ind~$x$ be the number of cosets in the double coset $KxK$.  Then the Hecke algebra has (the Schur) basis 
 (ind $\gamma_j$)$e\gamma_je$ ($j = 1,\ldots,t$) with structure constants
$$ \mu_{ijk} = |K|^{-1} |K\gamma_i K \cap \gamma_k K \gamma_j^{-1} K|$$ \cite[(11.34)]{CR}.
\end{proof}
When $H$ is a symmetric algebra, such as the case of group algebras, then the ``corner'' algebra $eHe$ is also a symmetric algebra (show the restricted nondegenerate trace is still nondegenerate,\cite[p.\ 456]{Lam2}). Of course, the existence of an idempotent integral in a finite-dimensional Hopf algebra $R$ is equivalent to $R$ being a semisimple algebra, in which case the depth is finite, but still an interesting value. 

In general, it would be nice to use Pareigis's Theorem on Tower~(\ref{eq: tower}) that a finite projective extension of symmetric algebras is a Frobenius extension; however, doing so without assuming $H$ is semisimple is a difficult problem.  The next proposition, lemma and theorem makes in-roads using the following strategy: assumptions on a ring extension $B \rightarrow A$ lead to conclusions about the natural inclusion $\End M_A \into \End M_B$ for certain modules $M_A$.  For example, if $B \rightarrow A$ is a separable extension, then for any module $M$, the endomorphism ring extension is split (by the trace map \cite{NEFE}).  If $M$ is a $B$-relative projective $A$-module, and $B \rightarrow A$ is a Frobenius extension, then   the trace mapping $\End M_B \rightarrow \End M_A$ is surjective.

\begin{prop}
\label{prop-progen}
  If $P$ is a progenerator $A$-module, and $B \into A$ is a Frobenius extension with surjective Frobenius homomorphism, then $\End P_A \rightarrow \End P_B$ is a Frobenius extension.  
\end{prop}
\begin{proof}
The hypothesis of surjectivity is equivalent to: $A_B$ (or ${}_BA$) is a generator \cite[Lemma 4.1]{LK2012}. We have elaborated on Miyashita's theory of \textit{Morita equivalence of ring extensions} in  \cite[Section 5]{K2016}, where it is shown that Frobenius extension is an invariant notion of this equivalence.   
There is a module $Q_A$ and $n \in \N$ such that \begin{eqnarray} A_A \oplus Q_A & \cong & n \cdot P_A \\
A_B \oplus Q_B & \cong & n \cdot P_B
\end{eqnarray}
Both $A_B$ and $P_B$ are progenerators. The endomorphism rings of the display equation leads to the following, with inclusion downarrows.
\begin{eqnarray*}
M_n(\End P_A) & \cong & \left( \begin{array}{cc}  
                                           A & \Hom (Q_A,A_A) \\
                                           Q & \End Q_A \end{array} \right)  \\
\downarrow \ \ \ \ \  & & \ \ \ \ \ \ \ \ \ \ \  \downarrow \\
M_n(\End P_B) & \cong &  \left( \begin{array}{cc}  
                                           \End A_B & \Hom (Q_B,A_B) \\
                                           \Hom (A_B, Q_B) & \End Q_B \end{array} \right) 
\end{eqnarray*}
The inclusions are the obvious ones including $A \into \End A_B$ given by left multiplication of elements of $A$, a Frobenius extension by the Endomorphism Ring Theorem (cf.\ \cite{NEFE}).  Since the matrix ring extension $M_n(\End P_A) \into
M_n(\End P_B )$ is Morita equivalent to $\End P_A \into \End P_B$ \cite[Example 5.1]{K2016}, it suffices to show that the matrix inclusion displayed above is itself Morita equivalent to the upper left``corner'' ring extension, which of course is Frobenius.  Define  the full idempotents $e =  \left( \begin{array}{cc}  
                                           1_A & 0 \\
                                           0 & 0 \end{array} \right)$ and $f =  \left( \begin{array}{cc}  
                                           \id_A & 0 \\
                                           0 & 0 \end{array} \right)$, which satisfy Morita equivalence  condition \cite[Prop.\ 5.3(2)]{K2016}, since $$\Hom (Q_A, A_A) \otimes_A \End A_B \cong \Hom (Q_B, A_B)$$ from the hom-tensor adjoint relation.  Of course $eM_n(\End P_A)e \cong A$ and $fM_n(\End P_B)f \cong \End A_B$. 
The idempotents are full, e.g. $$M_n(\End P_A) e M_n(\End P_A) = M_n(\End P_A),$$ since $Q_A$ is  projective. The idempotent $f$ is full since $Q_B$ and $A_B$ are progenerators and therefore similar modules, whence the composition mapping 
$$\Hom (A_B,Q_B) \otimes_{\End A_B} \Hom (Q_B,A_B) \rightarrow \End Q_B$$
and its reverse are surjective mappings.  
\end{proof}

\begin{lemma}
\label{lemma-integralformula}
If $H$ is a Hopf algebra and $M_H$ is a module, then $M_. \otimes H_. \cong M \otimes H_H$; for example,  the module $H$ satisfies the integral formula, $M \otimes H \cong (\dim M)H$ for any $M$ in the finite tensor category mod-$H$.  
\end{lemma}
\begin{proof}
For $m \in M$ and $h \in H$, map $m \otimes h \mapsto mh\1 \otimes h\2$ with inverse $m \otimes x \mapsto mS(x\1) \otimes x\2$.   
\end{proof}

Note that the comultiplication of a Hopf algebra $\cop: H \rightarrow H \otimes H$, or its coassociative power $\cop^n: H \rightarrow H^{\otimes (n+1)}$ may be viewed as a algebra monomorphism between Hopf algebras, where $H \otimes \cdots \otimes H$ (to any power) has the usual tensor Hopf algebra structure derived from $H$.  The comultiplication monomorphism is interesting from the point of view of endomorphism algebras of tensor powers of modules in the finite tensor category mod-$H$.  For example, if $R$ is a Hopf subalgebra and $Q$ their quotient module as before, then $Q^{\otimes n}_H$ is the pullback along this monomorphism of the cyclic module-coalgebra $Q^{\otimes n}_{H^{\otimes n}}$. However, the $\cop^n$ is not a Hopf algebra morphism, since it fails to commute 
with the coproducts in $H$ and $H^{\otimes n}$, and in addition, fails to commute with the antipodes $S$ and $S \otimes \cdots \otimes S$ (but does commute with $S^2$ and $S^2 \otimes \cdots \otimes S^2$). The following is then nontrivial, but easy to prove with ring theory. 

\begin{theorem}
If $H$ is unimodular, the algebra extension $\cop^n: H \into H^{\otimes (n+1)}$ is a Frobenius extension for any $n \in N$.   
\end{theorem}
\begin{proof}
We use Pareigis's Theorem in \cite[Prop.\ 5.1]{KSt}.  Since $H$ has a two-sided nonzero integral $t_H$,
$H \otimes H$ has as well $t_H \otimes t_H$, and also the higher tensor powers of $H$ are unimodular.  
The formula for Nakayama automorphism of $H$ \cite[Lemma 4.5]{KSt} involves $S^2$, and the modular function, which falls away since it is trivial for unimodular Hopf algebras: in conclusion, the Nakayama automorphism for $H^{\otimes n}$ is ${S^2}^{\otimes n}$ for each $n \geq 1$.  Note
then that the Nakayama automorphism stabilizes the image of $H$ under $\cop^n$.

Finally, we note that $(H^{\otimes (n+1)})_H$ is free, since taking $M = H^{\otimes n}$ in the Lemma yields
$(H^{\otimes (n+1)})_H \cong (\dim H)^n \cdot H_H$.  By Pareigis's Theorem, the algebra extension $\cop^n$ is a $\beta$-Frobenius extension with $\beta = (S^2 \otimes S^2)(S^{-2} \otimes S^{-2}) = \id_{\cop^n(H)}$, i.e., an ordinary Frobenius extension.  
\end{proof}
\begin{remark}
\begin{rm}
The stable category of a finite-dimensional algebra $A$ is the quotient category mod-$A$ modulo the ``ideal'' of projective homomorphisms between arbitrary $A$-modules. These are the $A$-module homomorphisms that factor through a projective $A$-module in mod-$A$ \cite{SY}.  Alternatively, these are homomorphisms that have  a matrix-like representation coming from projective bases.  For $A$ a Hopf algebra, \cite[Chen-Hiss]{CH} establish that the projective homomophisms are images of the trace mapping from plain linear mappings (and conversely).  As their proofs show, this generalizes to a Frobenius algebra $A$.  Introducing a subalgebra $B \subseteq A$, assuming $A$ is
a Frobenius extension of $B$ and replacing projectives with the larger subcategory of $B$-relative $A$-projectives, their proofs may also be modified to show that 
relative projective homomorphisms are precisely described by the images of trace map from
$\Hom(X_A,Y_A) \rightarrow \Hom (X_B,Y_B)$ (cf.\ \cite[Theorem 2.3, p.\ 95]{L}).  Modding out by the relative projective mappings yields then a $B$-relative stable category of potential interest.  For group algebras $B \subseteq A = kG$ with quotient $Q$, this is the category \textbf{stmod}${}_Q(kG)$  in the  recent paper \cite{La} in modular representation theory.   
\end{rm}
\end{remark}
\begin{remark}
\begin{rm}
The viewpoint of interior algebras and induced algebra of Puig and \cite{CAT} is  related to the approach of the endomorphism ring tower of a Frobenius extension $A \subseteq B$ with Frobenius homomorphism $E: A \rightarrow B$ and dual bases tensor $\sum_i x_i \otimes_B y_i$ (see \cite[Section 4.1]{LK2012} for a complete set of tower equations).  Then (*) $\End A_B \cong A \otimes_B A$ via $\alpha \mapsto \sum_i \alpha(x_i) \otimes_B y_i$ with inverse $a \otimes_B a' \mapsto aE(a'-)$.  The induced ``E-multiplication'' is given by $(a \otimes_B a')(x \otimes_B y) = aE(a'x) \otimes_B y$ with  $1$ the dual bases tensor.  The natural algebra extension $ A \into A \otimes_B A$ is itself a Frobenius extension with Frobenius homomorphism $E_1(x \otimes_B y) = xy$ with dual bases tensor $\sum_i x_i \otimes_B 1_A \otimes_B y_i$: the so-called Endomorphism Ring Theorem \cite{NEFE}.  Using (*) once again and 
after some elementary tensor cancellations, we arrive at 
\begin{eqnarray}
\label{eq: one}
\End A \otimes_B A_A & \cong & A \otimes_B A \otimes_B A \\
\label{eq: two}
\beta & \longmapsto & \sum_i \beta(x_i \otimes_B 1_A) \otimes_B y_i \\
\label{eq: three}
(x \otimes_B y \mapsto a \otimes_B a'E(a''x)y) & \leftarrow & a \otimes_B a' \otimes_B a'' 
\end{eqnarray}
These are algebra isomorphisms with respect to the $E_1$-multiplication given by 
\begin{equation}
\label{eq: four}
(u \otimes_B v \otimes w) (x \otimes_B y \otimes_B z) = u \otimes_B vE(wx)y \otimes_B z
\end{equation}
(compare \cite[(4.8)]{LK2012}).  For example, if $A$, $B$ are the group algebras
of a subgroup-finite group pair,  the formula is that of Puig's \cite[Example 2.1]{CAT}.  
Now if $B \rightarrow C$ is an algebra homomorphism, there is a natural $B$-$C$-bimodule structure on $C$ induced from this, making $C$ a $B$-interior algebra;
e.g. $C$ might be $eA$, or $eB$ where $e$ is a central idempotent in $A$, or $B$.  Linkelman's induced algebra is  $\Ind_A(C) := \End A \otimes_B C_C$,
an $A$-interior algebra.  In fact, the Eqs.~(\ref{eq: one})-(\ref{eq: four}) generalize completely straightforwardly to an algebra isomorphism, 
\begin{equation}
\label{eq: interior-inducedalgebras}
\End A \otimes_B C_C \cong A \otimes_B C \otimes_B A
\end{equation}
where the mappings and multiplication are omitted (just let $v,y,a' \in C$ and replace $1_A$ with $1_C$). Notice how (*) is recovered with $C = B$.  The paper \cite{CAT} shows how Eq~(\ref{eq: interior-inducedalgebras}) extends to a $\beta$-Frobenius extension $A \supset B$ such as a general Hopf subalgebra pair $H \supset R$.   
\end{rm}
\end{remark}
\begin{remark}
\begin{rm}
In \cite{SP} and two fundamental references  from 2011 therein, the notion of separable equivalence is an interesting weakening of Morita equivalence, which still preserves homological properties of rings or finite-dimensional algebras: for example,  finite generation of Hochschild cohomology, complexity, QF conditions, noetherianess, polynomial identity rings, global dimension,  representation type of finite dimensional algebras and more (compare with \cite[Theorem 6.1]{LK95}).
Two rings $A$ and $B$ are \textit{separably equivalent} (or split separably equivalent \cite[Prop.\ 6.3]{LK95}) if there are bimodules and one-sided progenerators ${}_AP_B$ and ${}_BQ_A$ such that ${}_AA_A \| P \otimes_B Q$
and ${}_BB_B \| Q \otimes_A P$ via  split epis $\nu$ and $\mu$, respectively.  They are more strongly \textit{symmetric separably equivalent}
(or finite separably equivalent \cite[Def.\ 6.1]{LK95}) 
if $P \otimes_B -$ and $Q \otimes_A -$ are adjoint functors in either order between categories of modules, i.e. Frobenius functors \cite[2.2]{NEFE}.
Symmetric separable equivalence provides extra structure such as, among other things,  $P \otimes_BQ$
is a ring with $\mu$-multiplication (for example, the E-multiplication in \cite[3.1]{NEFE}), isomorphic to $\End P_B$, which extends $A$ as a Frobenius extension; symmetrical statements are valid for the Frobenius homomorphism $\mu: Q \otimes_A P \rightarrow B$. The paper \cite{LK95} gives 
a half-dozen examples of symmetric separably equivalent algebras stemming from classical separable extensions. Note that the centers of equivalent algebras are not necessarily isomorphic, unlike for Morita equivalence; for example, any two finite-dimensional semisimple algebras over $\C$ are symmetric separably equivalent by the remarks following Eq.~(\ref{eq: tower}). 
\end{rm}
\end{remark}
\subsection{Ring-theoretic point of view}  Given a ring $A$ with right ideal $I$, the cyclic $A$-module
$A/I$ has endomorphism ring $\End (A/I)_A \cong B/I$ where $B$ is the \textit{idealizer} of $I$ in $A$,
the largest subring of $A$ in which $I$ is a two-sided ideal.  The details are in the lemma below.
\begin{lemma}
\label{lemma-idealizer}
The endomorphism ring of a cyclic module is determined by the idealizer of the annihilator ideal of a generator. 
\end{lemma}
\begin{proof}
Using the notation above, define the idealizer of $I$ in $A$ to be $B : = \{ x \in A \| \, xI \subseteq I \}$.  Clearly $B$ is a subalgebra of $A$ containing $I$ as a two-sided ideal, and maximal subalgebra in which $I$ is $2$-sided by its definition. Given $\alpha \in \End (A_I)_A$, consider $\alpha(1 + I): = b + I$,
which satisfies $\alpha(1 + I)I = I = (b + I)I = bI + I$, so $bI \subseteq I$ and $b \in B$.  
Conversely, left multiplication by each $b \in B$ is a well-defined endomorphism of $A/I$.  
\end{proof}

Applied to a Hopf subalgebra $R \subseteq H$ and $Q = H/R^+H$, the idealizer of the right ideal $R^+H$ in $H$ is $T = \{ h \in H \| \, hR^+H \subseteq R^+H \}$.  The idealizer $T$ contains for example $R^+H$, $R$, $Ht_R$, $t_H$ (where $t$ denotes nonzero integral as usual) 
and the centralizer of $R$ in $H$, but is not a Hopf subalgebra, nor left or right coideal subalgebra of $H$ in general. Since $\End Q_H \cong T/R^+H$, there is a monomorphism of $\End Q_H \into Q$
given by $\alpha\mapsto \alpha(\overline{1_H})$, whose image is a subalgebra in the coalgebra $Q$ (!). This map is a surjective if and only if $R$ is a normal Hopf subalgebra in $H$, summarized in the lemma, with proof left as an exercise.
\begin{lemma}
The evaluation mapping $\End Q_H \rightarrow Q$ is an isomorphism if and only if $R^+H = HR^+$.
\end{lemma}

See Example~\ref{ex-eightdim} below for a computation of an idealizer $T$ of $Q$ for the $8$-dimensional small quantum group and the $4$-dimensional Sweedler Hopf subalgebra (which is not normal).  The idealizer subalgebra associated to the coset $G$-space $Q$ in Lemma~\ref{lem-Hecke} is computed in terms of a Schur basis in \cite[12.10]{L}.  Since $Q^{\otimes n}$  is isomorphic to  the quotient module of the Hopf subalgebra pair $R^{\otimes n} \subseteq H^{\otimes n}$,  Lemma~\ref{lemma-idealizer}  applies to show that 
\begin{equation}
\End Q^{\otimes n}_{H^{\otimes n}} \cong T^{\otimes n}/ ({R^{\otimes n}}^+ ) H^{\otimes n}.
\end{equation}

%%%%%%%%%%%%%%%%%%%%%%%%%%%%%%%%%%%%%%%%%%%%%%

\section{Trace ideals of tensor powers of Q}
\label{section-traceideal}

Recall that for any ring $R$ and module $M_R$, the trace ideal $$\tau(M_R) = \{ \sum_i f_i(m_i) \| \, f_i \in \Hom (M_R, R_R), m_i \in M \}.$$  Note then that $N_R \| M_R$
implies $\tau(N_R) \subseteq \tau(M_R)$.  Hence, $\tau(N_R) = \tau(M_R)$ is a necessary condition for $N_R \sim M_R$. Recall that the trace ideal $\tau(M_R) = R$ if and only if $M_R$ is a generator.  Recall
that generators are faithful, and conversely if $R$ is a QF  (e.g. Frobenius or Hopf) algebra and $M$ finitely generated.  

Let $R \subseteq H$ be a Hopf subalgebra in a finite-dimensional Hopf algebra, $Q = Q^H_R$ the
quotient module in mod-$H$ defined above, and $t_R$ a nonzero right integral in $R$. 
\begin{prop}
The trace ideal $\tau(Q_H) = Ht_R H$.
\end{prop}
\begin{proof}
Follows from \cite[Prop.\ 3.10]{K2016} since  $Ht_R \stackrel{\sim}{\longrightarrow} \Hom (Q_H, H_H)$
via left multiplication by $ht_R$ for each $h \in H$.  
\end{proof}

\begin{example}
\label{ex-eightdim}
\begin{rm}
Consider $H =   \overline{U_q}(sl_2(\C))$ as in Example~\ref{example-smallqg} at the $4$'th root of unity $q = i$, which is the $8$ dimensional algebra  generated  by $K,E,F$ where $K^2 = 1$, $E^2 = 0 = F^2$, $EF = FE$, $KE = -EK$, and $KF = -FK$.  Let $R$ be the Hopf subalgebra of dimension $4$ generated by $K,E$ (isomorphic to the Taft algebra).  Then a calculation shows that 
$Q$  is spanned by $\overline{1}$ and $\overline{F}$, that $t_R = E(1 + K)$, and
$$Ht_RH = \C t_R + \C EF + \C EFK $$
a 3-dimensional ideal containing (rad $H)^2$.

Also note that $\Ann Q_H = EH$ (a Hopf ideal), since $\overline{F}E = \overline{EF} = 0$.  Then  $\ell_Q = 1$, but we cannot apply Theorem~\ref{theorem-ell} to conclude $d_h(R,H) = 3$ since $H$ is not semisimple. 
The ordinary depth satisfies $3 \leq d(R, H) \leq 4$ by a computation similar to \cite[Example 1.10]{HKS}, which implies that $d_h(R,H) = 3$ or $5$.  My calculations indicate that $Q \not \sim Q\otimes Q$ as $H$-modules, implying $d_h(R,H) = 5$. 

In addition, the idealizer $T$ is $7$-dimensional spanned by $R$, $(1+K)EF$, $F(1+K)$ and  $(1-K)EF$, a subalgebra which is not a left or right coideal subalgebra.  Since $R^+H = EH + (1-K)H$ is $6$-dimensional, $\End Q_H$ is one-dimensional, which may also be verified directly.
\end{rm}
\end{example}

\subsection{Ascending chain of trace ideals of tensor powers of $Q$} Since
the tensor powers of $Q$ satisfy $Q^{\otimes m} \oplus * \cong Q^{\otimes (m+1)}$ as $H$-modules for each integer $m \geq 1$, it follows that their trace ideals satisfy 
$\tau(Q^{\otimes m}) \subseteq \tau(Q^{\otimes (m+1)})$. Let $L_Q$ denote
the length of the ascending chain,
\begin{equation}
\tau(Q) \subset \tau(Q^{\otimes 2}) \subset \cdots \subset \tau(Q^{\otimes L_Q}) :=  \mathcal{I}
\end{equation}
 necessarily finite since $H$ is finite-dimensional.  If $\mathcal{I} = H$, then
$Q^{\otimes L_Q}$ is faithful, implying that $\Ann Q^{\otimes L_Q} = 0$. Then 
the length $\ell_Q$ of the descending chain of 
$\{ \Ann Q^{\otimes n}\}_{n \in \N}$ satisfies $\ell_Q \leq L_Q$. Conversely,
if $\mathcal{I} = H$, and $\Ann Q^{\otimes \ell_Q} = 0$, then $\tau(Q^{\otimes \ell_Q}) = H$, which shows that $L_Q \leq \ell_Q$.  Recall from \cite{HKY} that
an $H$-module $W$ is conditionally faithful, if one of its tensor powers is faithful.
\begin{prop}
If $Q$ is conditionally faithful, then $L_Q = \ell_Q$.   
\end{prop}
When computing depth for general $Q$, with nontrivial maximal Hopf ideal $I$ in $\Ann Q_H$,  we recall \cite[Lemma 1.5]{HKS} implying that the depth of $Q_H$ is  equal to the depth of $Q$ as a (conditionally faithful) $H/I$-module.  Thus the proposition is useful in this situation as well.  

\section{Minimal polynomials of $Q$ in $A(R)$ and $A(H)$}

In this section, we take $R \subseteq H$ to be a Hopf subalgebra pair of \textit{semisimple} Hopf algebras over an algebraically closed field $k$ of characteristic zero.  We identify the quotient module $Q = Q^H_R$ with its isoclass in the Green ring $A(H)$,
equal to $K_0(H)$ since $H$ is semisimple; $A(H)$ has basis of simple $H$-modules.  The restriction $Q_R$ also represents an isoclass in $A(R) = K_0(R)$.  If $Q_R$ satisfies a minimum polynomial $m(X) = 0$ in $A(R)$,  we will note in this section that $Q_H$ satisfies $Xm(X) = 0$, in most cases a minimum polynomial equation in $A(H)$.

Define a linear endomorphism $\mathcal{T}: A(R) \rightarrow A(R)$ by $\mathcal{T}(X) = X \uparrow^H \downarrow_R$
for every $R$-module $X$ and its isoclass. Similarly we define a linear endomorphism $\mathcal{U}: A(H) \rightarrow A(H)$ by restriction followed by induction.  By Eq.~(\ref{eq: fundamental}), there is a natural isomorphism $\mathcal{U}(Y) = Y \otimes_R H \cong Y \otimes Q$, i.e., the right multiplication by $Q$, an operator in $\End A(H)$
represents the endofunctor $\mathcal{U}$.  By an induction argument, a polynomial $p(X)$ in the powers of $\mathcal{U}$ are then given
by $p(\mathcal{U}) = p(Q)$  in $\End A(H)$, where the tensor powers of $Q$ are again identified with their right multiplication operators on the Green ring $A(H)$.

Let $M$ denote the matrix $K_0(H) \rightarrow K_0(R)$ of restriction relative to the bases of simples; i.e., for each $H$-simple $U_j$ ($j = 1,\ldots,q$), express its restriction $U_j\downarrow_R = \sum_{i=1}^q m_{ij} V_i$, where $V_i$ ($i=1,\ldots,p$) are the $R$-simples and $m_{ij}$
the nonnegative integer coefficients of $M$, a $p \times q$-matrix.  Since $M$ is derived from restriction of modules, each column of $M$ is nonzero.

\begin{example}\label{example-perm}
\begin{rm}
Let $H = \C S_3$, the symmetic group algebra isomorphic to $\C \oplus M_2(\C) \oplus \C$, and $R = \C S_2 \cong \C \oplus \C$, embedded by fixing one letter.  The restiction of $H$-simples is well-known (e.g. \cite{FH}) to be given by the (weighted) bipartite graph as follows
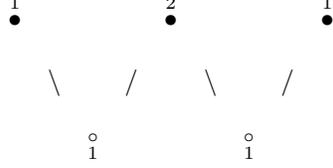
\begin{equation}
\label{eq: invertedW}
\begin{array}{rcccccccl}
 \stackrel{1}{\bullet} & & &&\stackrel{2}{\bullet} & & & &\stackrel{1}{\bullet} \\
&&&&&&&& \\
& \setminus & & / & & \setminus & &  /  & \\
&&&&&&&& \\
 & & \stackrel{\circ}{\mbox{\scriptsize 1}} & & & & \stackrel{\circ}{\mbox{\scriptsize 1}}  & &
\end{array} $$
with $M$ as the incidence matrix (from left to right) 
 $$ M = \left( \begin{array}{ccc}
1 &  1 & 0  \\
0 &  1 & 1 
\end{array}
\right) \end{equation}

\end{rm}
\end{example}

Since restriction of $H$-modules to $R$-modules and induction of $R$-modules to $H$-modules are adjoint functors (e.g., \cite{NEFE}), the transpose $M^t$
represents the linear mapping $K_0(R) \rightarrow K_0(H)$ given by
$V_i \otimes_R H = \sum_{j=1}^q m_{ij}U_j$.  
In other words, the columns of $M$ show the restriction of a top row of $H$-simples, and the rows of $M$ show induction of the bottom row of $R$-simples at the same time, in an incidence diagram drawn as a  weighted bicolor multi-graph  of any inclusion
of subalgebra pairs of semisimple $k$-algebras (see \cite{GHJ} for the exact details).  Thus, each row (and each column) of $M$ is nonzero.  

It follows that the linear endomorphism $\mathcal{T} \in \End_k A(R)$ has matrix representation $B = MM^t$, a symmetric  matrix of order $p$ (with nonzero diagonal elements). Thus, $B$ has a full set of $p$ eigenvalues.   In these terms, the matrix of $\mathcal{U}$ relative to the bases of simples $\{ U_1,\ldots,U_q \}$ of $A(H)$ is given by $C := M^t M$
(cf.\ \cite[Eq.~(10)]{HKS}). As a consequence of Eq.~(\ref{eq: fundamental}) and its iterations, a minimum polynomial of $C$ is
also a minimum polynomial of the isoclass of $Q$ in the Green algebra $A(H)$.

Let $G$ be a finite group, and  $\mbox{Cl}(G)$ denote the set of conjugacy classes of $G$. For group algebras, we recall

\begin{theorem}\cite[Theorem 6.16]{BKK}
Let $k H \subseteq k G$ be the group algebras of a subgroup pair $H \leq G$. The nonzero eigenvalues of $B$ are
$$\mathcal{E} := \{ \frac{|G|}{|H|} \frac{|C \cap H|}{|C|}: \ C \in \mbox{Cl}(G), C \cap H \neq \emptyset \}$$ Note the Perron-Frobenius eigenvalue $|G:H|$.  If $t = |\mathcal{E}|$, then  the degree of the minimum polynomial of $B$ is $t$ or $t+1$, and the minimum depth $d_0(H,G) \leq 2t+1$.  
All eigenvalues of $B$ are nonzero iff each conjugacy class of $G$ restricts to only one conjugacy class of $H$: in this case, the degree of the minimum polynomial of $B$ is $t$ and the minimum depth $d_0(H,G) \leq 2t-1$.  
\end{theorem}

For example, when $G = S_3$ and $H = S_2$ as in Example~\ref{example-perm},
the conjugacy classes of $S_3$
have representatives $(1), (12), (123)$.   
From the inclusion matrix $M$ we obtain $$B =  \left( \begin{array}{cc}
2 &  1   \\
1 &  2 
\end{array}
\right)  \ \ \ 
C = \left( \begin{array}{ccc}
1 & 1 & 0 \\
1 & 2 & 1 \\
0 & 1 & 1 
\end{array}
\right) 
$$ where $B$ has minimum polynomial $(X-1)(X-3)$ and $C$ has minimal polynomial $X(X-1)(X-3)$.  The depth is computed to
be $d_0(S_2,S_3) = 3$ in \cite{BuK} (and $d_0(S_{n-1},S_n) = 2n-1$ in \cite{BKK}, $d_h(S_{n-1},S_n) = 2n+1$ \cite{K2014}).

\begin{lemma} 
Suppose $M$ is a $p \times q$ real matrix.  If $m(X)$ is a minimum polynomial of $MM^t = B$, then $C := M^t M$ satisfies $Cm(C) = 0$. If $p < q$
and $M$ represents a surjective linear mapping $k^q \rightarrow k^p$, then
$B$ is a linear automorphism of $k^p$ and $C$ has minimum polynomial $Xm(X)$.  
\end{lemma}
\begin{proof}
Note that $M^t B M = C^2$, and $M^t B^n M = C^{n+1}$ by a similar inductive step.  Then $0 = M^t m(B) M = Cm(C)$.

Recall that the standard inner product satisfies $\bra Mx, y \ket = \bra x, M^t y \ket$ for $x \in k^q, y \in k^p$,
since $M$ has real coefficients.  If $M^t y = 0$, then $\bra Mx,y \ket = 0$
for all $x \in k^q$.If $M$ is surjective, $y = 0$, and hence $M^t$ represents
an injective linear transformation of $k^p \into k^q$.  If $MM^t x = 0$, then
\begin{equation}
\label{eq: barry}
 0 = \bra MM^t x, x \ket = \bra M^t x, M^t x \ket .
\end{equation}
It follows that $M^tx = 0 $.  Since $M^t$ is injective, $x = 0$.  Then $MM^t$ is
nonsingular.  

If $p < q$, note that $C = M^tM$ represents a linear endomorphism of $k^q$ that factors through a space of lesser dimension , thus $C$ has nonzero kernel and eigenvalue $0$.  
Since $B,C$ are symmetric matrices, they are diagonalizable.  Since $B$ has nonzero determinant, it has positive eigenvalues (positive by a computation like in Eq.~(\ref{eq: barry}).  The matrix $C$ has the same eigenvalues as $B$ as well as $0$, since $Cm(C) = 0$ \cite[Theorem 10]{HK}.  
\end{proof}
In the situation of $M$ an inclusion matrix of semisimple group or Hopf algebras, 
the restriction of $A(H) \rightarrow A(R)$ is often surjective and $\dim A(R) = p < \dim A(H) = q$.  In this case, the symmetric matrix $B$ has only nonzero eigenvalues   and  $Xm(X)$ is a minimum polynomial of $C$.  We summarize:
\begin{theorem}
If $m(X)$ is a minimum polynomial of $MM^t$, then $Q_H$ has  minimum polynomial $m(X)$ or $Xm(X)$ in  $A(H)$. If the number of nonisomorphic $R$- and $H$- simples $p < q$ and the inclusion matrix $M$
is surjective, then $Q_H$ has minimum polynomial $Xm(X)$.  
\end{theorem}

\begin{example}
\begin{rm}
Consider the alternating groups $A_4 < A_5$ with inclusion matrix 
 $$M = \left( \begin{array}{ccccc}
1 &  1 & 0 & 0 & 0  \\
0 & 0  & 1 & 0 & 0 \\
0 & 0 & 1 & 0 & 0 \\
0 & 1 & 1 & 1 & 1 
\end{array}
\right) $$
which has rank $3$, so does not represent a surjective linear transformation.  The matrix
$MM^t$ has the eigenvalues 0, 1, 2, and 5, as does the matrix $M^tM$, with equal minimum polynomials.  The depth $d_0(A_4,A_5) = 5$ since $(MM^t)^2 > 0$, and
the h-depth $d_h(A_4,A_5) = 5$ since similarly $(M^tM)^2 > 0$.  
\end{rm}
\end{example}
The material in this section is evidence for a conjecture that 
$d_0(H,G) \leq d_h(H,G)$ for any subgroup pair $H \leq G$ of finite groups.  

\subsection{McKay quiver of $Q$}  The matrix $C$ formed above from the inclusion matrix $M$ of a semisimple Hopf subalgebra pair $R \subseteq H$ is also the matrix of adjacency of the McKay quiver of
$Q$.  For a general module $V_H$ with character $\chi_V$, its McKay quiver has $q$ vertices for each $H$-simple $U_i$ with irreducible character $\chi_i$, the weighted edges of the quiver are given
by the positive integers among the nonnegative integers $a_{ij}$ defined by $\chi_i \chi_V = \sum_{j= 1}^q a_{ij} \chi_j$ ($i = 1,\ldots,q$) \cite{BBH}.   Then applied to $V = Q$ it is an exercise to check that the $q \times q$-matrix $(a_{ij}) = C$, since the character of $U_i \otimes_R H \cong U_i \otimes Q$ is $\chi_i \chi_V$.  

For example, for $R = \C S_2 \subseteq H = \C S_3$ considered above, the coefficient $(i,j)$ of matrix $C$ record the number of walks of length two from the simple $U_i$ to the simple $U_j$ in the top row of black vertices in the graph of Example~\ref{example-perm}.  The zeroes in $C$ record that there are no walks from vertices 
$1$ to $3$ shorter than length $4$.  The matrix coefficient $c_{22} = 2$ records two walks of length
two from the middle vertex, one to the left and one to the right.  Note that $C^2$ is strictly positive
and records the number of walks of length $4$ between the vertices. Continuing like this, one may read off the h-depth of a Bratteli diagram from adding one to the longest walk between black vertices, in a manner similar to the graphical method applied to the white vertices for finding subgroup depth in \cite{BKK}.  

If $V = U_k$ one obtains the fusion coefficients in  $\chi_i \chi_k = \sum_{j=1}^q N^k_{ij} \chi_j$.  The monograph \cite[ch.\ 5]{KSZ} proves several interesting theorems about the matrix of nonnegative
integers $\mathcal{A} := (a_{ij})$ defined by tensoring the simples by a module $V$.  We record this information  in the next proposition. Recall that a square matrix $X$ is indecomposable if the basis may not be permuted to obtain an equivalent matrix with zero corner block.  

\begin{prop}\cite{KSZ}
The maximal Hopf ideal  $\Ann V_H$ is zero if and only if $\mathcal{A}$ is indecomposable.
The Perron-Frobenius eigenvalue of $\mathcal{A}$ is $\dim V$.  Consequently, the Hopf core of
$R \subseteq H$ is zero if and only the order $q$ matrix $C = M^tM$ is indecomposable; the Perron-Frobenius eigenvalue of $C$ is $\frac{\dim H}{\dim R}$.  
\end{prop}
 \begin{example}
\begin{rm}
Consider the dihedral group of eight elements $D_8 < S_4$.  Both groups have five conjugacy classes.  The subgroup core is equal to
$\{ (12)(34), (13)(24), (14)(23), (1) \}$.  The bicolored graph
of the inclusion  has two connected components, one a W as in Fig.~(\ref{eq: invertedW}), and another an inverted W.   The inclusion matrix is the decomposable matrix,
 $$M = \left( \begin{array}{ccccc}
1 &  1 & 0 & 0 & 0  \\
0 & 1  & 1 & 0 & 0 \\
0 & 0 & 0 & 1 & 0 \\
0 & 0 & 0 & 1 & 1 \\
0 & 0 & 0 & 0 & 1 
\end{array}
\right) $$
The symmetry of the bicolored graph, or computation with pencil and paper, reveal that $MM^t$ and $M^tM$ are equivalent matrices, therefore with equal minimum polynomials. 
The minimum depth is $d_0(D_8,S_4) = 4$,  h-depth $d_h(D_8,S_4) = 5$
(and its minimum odd depth $d_{odd}(D_8,S_4) = 5$) .  
\end{rm}
\end{example}

%%%%%%%%%%%%%%%%%%%%%%%%%%%%%%%%%%%%%%%%%%%%%

\end{document}